\documentclass[12pt,reqno]{amsart}
\ifx\pdfsuppressptexinfo\undefined\else\pdfsuppressptexinfo=-1\fi

\usepackage{amsmath,amsthm,amssymb}
\usepackage{enumitem}
\usepackage[hmargin=1.15in,vmargin=1.15in]{geometry}
\usepackage[colorlinks,linkcolor=blue,citecolor=blue,urlcolor=blue]{hyperref}
\usepackage{bookmark}
\setlist[enumerate,1]{label=(\arabic*),font=\normalfont}
\setlist[enumerate,2]{label=(\alph*),font=\normalfont}
\numberwithin{equation}{section}

\newcommand{\F}{{\mathbb F}}
\newcommand{\Fcal}{{\mathcal F}}
\newcommand{\Tcal}{{\mathcal T}}

\DeclareMathOperator{\Aut}{Aut}
\DeclareMathOperator{\Gal}{Gal}
\DeclareMathOperator{\PGU}{PGU}

\theoremstyle{plain}
\newtheorem{theorem}{Theorem}[section]
\newtheorem{corollary}[theorem]{Corollary}
\newtheorem{proposition}[theorem]{Proposition}
\newtheorem{lemma}[theorem]{Lemma}
\newtheorem*{theoremA}{Theorem~A}
\newtheorem*{theoremB}{Theorem~B}
\newtheorem*{theoremC}{Theorem~C}

\theoremstyle{definition}
\newtheorem{definition}[theorem]{Definition}

\theoremstyle{remark}
\newtheorem{remark}[theorem]{Remark}

\begin{document}

\title[Fixed subfields of the Hermitian function fields]
{Explicit equations of Galois subfields of Hermitian function fields with respect to decomposition groups}
\date{}

\author[Liming Ma]{Liming Ma}
\address{School of Mathematical Sciences, University of Science and Technology of China,
Hefei, Anhui 230026, China}
\email{lmma20@ustc.edu.cn}

\author[Yipeng Wang]{Yipeng Wang}
\address{School of Mathematical Sciences, University of Science and Technology of China,
Hefei, Anhui 230026, China}
\email{jisunxingfu@mail.ustc.edu.cn}


\begin{abstract}
Let $q$ be a prime power and $\mathbb{F}_{q^2}$ be the finite fields of $q^2$ elements. The Hermitian function field $H=\mathbb{F}_{q^2}(x,y)$ defined by
$y^q+y=x^{q+1}$ is a well-known maximal function field with the largest possible genus.  Let $A(P_\infty)$
be the decomposition group of the infinity place $P_\infty$ of $H$ which is the common pole of $x$ and $y$.
For every subgroup $G\le A(P_\infty)$, we construct explicit generators
of Galois subfield $H^G$ of $H$ with respect to $G$ and determine an absolutely irreducible equation defining the smooth affine
plane model for such a Galois subfield.  For $p$-subgroups, the generators can be chosen so that the
defining equation has an additive polynomial on the left-hand side and an $\mathbb{F}_p$-quadratic polynomial on the
right-hand side.  For $q=27$, we can construct a
genus-two subfield $D\subset H$ that is not isomorphic to $H^J$
for any subgroup $J\le\Aut(H)$ from the explicit equations of Galois subfields of the Hermitian function field.
To the best of our knowledge, this is the first example of a maximal function field
covered but not Galois-covered by the same Hermitian function field.
\end{abstract}

\maketitle

\section{Introduction}
\label{sec:intro}

Let \(q=p^f\) be a prime power and $\F_q$ be the finite fields of $q$ elements.
Let $F/\mathbb F_q$ be an
algebraic function field of one variable with the full constant field $\mathbb F_q$ of genus $g$.
By the celebrated the Hasse--Weil theorem, the number $N(F)$ of rational places of $F/\mathbb F_q$ is bounded by
$$q+1-2g\sqrt{q}\le N(F)\leq q+1+2g\sqrt{q}.$$
If $N(F)= q+1+2g\sqrt{q}$, then
 $F$ is said to be maximal. It is clear that $F$ can be maximal only if either $g$ is zero or $q$ is a square.
 Finding maximal function fields with new genera or explicit equations are very important, not only from the number theoretic perspective but also for applications in coding theory and cryptography.

The most famous maximal function field is the
Hermitian function field $H =\mathbb F_{q^2}(x,y)$ defined by \(y^q+y=x^{q+1}\), since it has an extremely large automorphism group \(A=\Aut(H)\simeq \PGU(3,q)\) and the largest possible genus $q(q-1)/2$ for maximal function fields over $\mathbb F_{q^2}$ \cite{Ihara1981}.
Moreover, Hermitian function field is the unique maximal function field with genus $q(q-1)/2$ up to isomorphism \cite{RuckStichtenoth1994}.
For a maximal function field $F/\mathbb F_{q^2}$, Serre proved that any subfield $\mathbb F_{q^2} \subsetneq E \subseteq F$ is maximal as well \cite{Lachaud1987}.
Thus, any fixed subfield of a maximal function field with respect to subgroups of its automorphism group is maximal.
Such an approach to construct maximal function fields via Galois subfields of Hermitian function fields was initiated in
\cite{GarciaStichtenothXing2000} and further studied in \cite{CossidenteKorchmarosTorres2000,AbdonQuoos2004,BassaMaXingYeo2013,MontanucciZini2018,DallaVoltaMontanucciZini2019,MaXing2019,
MontanucciZiniNoFixed2020,MontanucciZini2020}. In particular, the genus of the Galois subfield of the Hermitian function field with respect to any subgroup of the decomposition group of the infinite place was determined by Garcia, Stichtenoth, and Xing in \cite{GarciaStichtenothXing2000}. By analyzing the subgroup structures of decomposition group $A(P_\infty)$, the genera of Galois subfields of Hermitian function fields with respect to subgroups of $A(P_\infty)$ can be determined completely in the odd characteristic case \cite{BassaMaXingYeo2013}.
However, the explicit equations of fixed subfields \(H^G\) of Hermitian function field with respect to \(G\le A(P_\infty)\) are not determined completely.

In \cite{CossidenteKorchmarosTorres2000}, explicit plane models were obtained for prime-degree Galois quotients of Hermitian function fields.
 Some maximal function fields with explicit equations given by additive polynomials were proven to be Galois subfields of the Hermitian function fields in \cite{GarciaKawakitaMiura2006,GarciaOzbudak2007,GarciaTafazolian2008}.
In \cite{GattiKorchmaros2024}, explicit equations of Galois subfields of the Hermitian function field were obtained for subgroups with order $p^2$.
In \cite{DionigiGatti2025}, the authors obtained explicit equations of all Galois covers of the
Hermitian curve with Galois group of order $dp$ where $p$ is the characteristic of $\F_{q^2}$ and $d\neq p$.

In this manuscript, we determine the generators and explicit defining equations of the
Galois subfields \(H^G\) of the Hermitian function field with respect to every subgroup \(G\le A(P_\infty)\) of the decomposition group.

\begin{theoremA}
For every subgroup \(G_0\le A(P_\infty)\), we construct explicit
functions \(T,W\in H\) and an absolutely irreducible polynomial
\(F_{G_0}\in\F_{q^2}[S,Y]\) such that
\[
        H^{G_0}=\F_{q^2}(T,W),\qquad F_{G_0}(T,W)=0.
\]
The polynomial \(F_{G_0}\) defines a smooth affine plane model of
\(H^{G_0}\).  The construction also gives, up to
isomorphism, the fixed field of every \(p\)-subgroup of \(A\).
For \(p\)-subgroups, the equation can be chosen with an additive
left-hand side and an \(\F_p\)-quadratic right-hand side.
\end{theoremA}

Let \(A_1(P_\infty)\) be the normal Sylow \(p\)-subgroup of \(A(P_\infty)\).
For a \(p\)-subgroup of \(A_1(P_\infty)\), subspace polynomials remove the
central translations, and a finite weighted sum solves the remaining
polynomial difference equations.  A construction using a basis of the
translation space gives the \(\F_p\)-quadratic right-hand side in
Theorem~\ref{thm:quadratic_additive_normal_form}.
For an arbitrary subgroup of \(A(P_\infty)\), an explicit conjugation makes
a cyclic complement to its normal Sylow \(p\)-subgroup act diagonally,
after which we take scalar invariants.
The scalar and abelian specializations recover the models in
\cite{CossidenteKorchmarosTorres2000} associated with prime-order subgroups
of \(A(P_\infty)\), as well as the pure scalar models in
\cite[Example~6.3]{GarciaStichtenothXing2000}.

We next describe the abelian \(p\)-subgroups more explicitly.  For
\(U\le A_1(P_\infty)\), let \(V_U\) be the set of translation
parameters in its action on \(x\), and let \(\Delta_U\) consist of the
\(\delta\) such that \(U\) contains the automorphism
\(x\mapsto x,\ y\mapsto y+\delta\).
Put \(\Tcal=\{\delta\in\F_{q^2}:\delta^q+\delta=0\}\), and for an
\(\F_p\)-subspace \(W\subseteq\F_{q^2}\), write
\(B_W(T)=\prod_{w\in W}(T-w)\).  For \(\F_p\)-subspaces
\(E\subseteq\F_q\) and \(\Delta\subseteq\Tcal\), let \(A_E,A_\Delta\)
be the monic additive polynomials determined by
\[
        A_E(B_E(T))=T^q-T,\qquad
        A_\Delta(B_\Delta(T))=T^q+T.
\]
If \(V_U=\{0\}\), then \(U\) is central and its fixed field has equation
\(A_{\Delta_U}(Z)=X^{q+1}\).  The noncentral abelian case is as follows.

\begin{theoremB}
Let \(U\le A_1(P_\infty)\) be noncentral and abelian.
\begin{enumerate}
\item If \(p\) is odd, there are a nonzero \(\F_p\)-subspace
\(E\subseteq\F_q\), a scalar \(c\in\F_q^*\), and an additive polynomial
\(R\in\F_{q^2}[T]\) of degree less than \(q\), with coefficients in
\(\Tcal\), such that, with \(\Delta=\Delta_U\), the field \(H^U\) is
isomorphic to a function field with equation
\[
        A_\Delta(Z)=cA_E(X)^2+R(A_E(X)).
\]

\item If \(p=2\), then \(V_U=\varrho E\) for some
\(\varrho\in\F_{q^2}^*\) and a nonzero \(\F_2\)-subspace
\(E\subseteq\F_q\), and \(\varrho^{q+1}e^2\in\Delta_U\) for every
\(e\in E\).  Conversely, for fixed \(\varrho,E\) and an
\(\F_2\)-subspace \(\Delta\subseteq\F_q\) satisfying
\(\varrho^{q+1}e^2\in\Delta\) for every \(e\in E\), there are exactly
\(\bigl(q/|\Delta|\bigr)^{\dim_{\F_2}E}\) abelian subgroups with
translation space \(\varrho E\) and central space \(\Delta\).
\end{enumerate}
\end{theoremB}

In characteristic two, the quadratic refinements of
Proposition~\ref{prop:char_two_quadratic_refinements} parameterize these
subgroups and give the data for the fixed-field equations in
Theorem~\ref{thm:quadratic_additive_normal_form}.
In odd characteristic, the central family and the subfamily \(R=0\)
realize every genus occurring among fixed fields of \(p\)-subgroups of
\(A\), including genus zero
(Corollary~\ref{cor:odd_triangular_realization}).
Section~\ref{sec:closed_families} also treats a family of unitary Heisenberg
subgroups.

We finally ask whether every intermediate function field
\(\F_{q^2}\subsetneq L\subseteq H\) is isomorphic to a Galois subfield \(H^J\)
for some \(J\le A\).
The possibility of a negative answer was highlighted in
\cite[Question~4.5]{BartoliMontanucciTorres2021}, and the general existence
problem was still described as open in
\cite[Remark~2.2]{GattiGhiandoniKorchmaros2025}.
The following result provide an explicit maximal function field covered but not Galois-covered by the same Hermitian function field.

\begin{theoremC}
Let \(H_{27}\) denote the Hermitian function field over \(\F_{27^2}\).  There
is a genus-two subfield \(D\subset H_{27}\) that is not isomorphic to
\(H_{27}^J\) for any subgroup \(J\le\Aut(H_{27})\).
The extension \(H_{27}/D\) is separable of degree \(108\) and is not Galois.
\end{theoremC}

The inclusion \(D\subset H_{27}\) is obtained by taking invariants of
an intermediate field from
Theorem~\ref{thm:odd_triangular_closed_model} under an order-four automorphism.

This paper is organized as follows. In Section~\ref{sec:pointed_subgroup_data}, we describe the subgroup data of the decomposition group of infinity place of the Hermitian function field. In Section~\ref{sec:subspace_calculus}, we collect the required properties of subspace polynomials. In Section~\ref{sec:wild_fixed_field}, we  construct explicit generators and defining equations for fixed fields of \(p\)-subgroups, including equations with an \(\mathbb F_p\)-quadratic right-hand side. In Section~\ref{sec:cyclic_scalar_part}, we extend the construction to arbitrary subgroups of \(A(P_\infty)\) by taking scalar invariants. In Section~\ref{sec:closed_families}, we specialize these results to abelian subgroups and a family of unitary Heisenberg subgroups. Finally, we provide the first maximal function field covered but not Galois-covered by the same Hermitian function field in Section~\ref{sec:non_galois_subcover}.

\section{The Sylow \texorpdfstring{\(p\)}{p}-subgroup and its subgroup data}
\label{sec:pointed_subgroup_data}

Let \(H=\F_{q^2}(x,y)\), where \(y^q+y=x^{q+1}\) be the Hermitian function field and let
\(A=\Aut(H)\) be the automorphism group of $H$.  Let \(P_\infty\) be the common pole of \(x\)
and \(y\) and it is easy to see that
\((x)_\infty=qP_\infty\) and \((y)_\infty=(q+1)P_\infty\).  For
\(G\le A\), write \(P_\infty^G\) for the restriction of \(P_\infty\) to
\(H^G\).

We follow \cite{Stichtenoth2009} for algebraic function fields, places,
constant-field extensions, and Galois fixed fields.  Automorphisms act on
functions on the left, so \(\rho(H^G)=H^{\rho G\rho^{-1}}\) for
\(G\le A\) and \(\rho\in A\).  The Hermitian generators \(x,y\) are fixed
throughout; the symbols \(X,Z,T,W\) are used for auxiliary generators and may
be redefined from one construction to another.

\subsection{The decomposition group and its Sylow subgroup}
\label{subsec:decomposition_sylow}

For \(a\in\F_{q^2}^*\) and \(b,c\in\F_{q^2}\), write \([a,b,c]\) for
the transformation
\[
        x\longmapsto ax+b,
        \qquad
        y\longmapsto a^{q+1}y+ab^qx+c.
\]
Such a transformation preserves the Hermitian equation precisely when
\(c^q+c=b^{q+1}\), and distinct triples define distinct
transformations.  By \cite[equation~(2.2)]{GarciaStichtenothXing2000}, every
element of \(A(P_\infty)\) has this form.  Hence
\[
        A(P_\infty)=\{[a,b,c]:c^q+c=b^{q+1}\}.
\]
Put \(\Tcal=\{\delta\in\F_{q^2}:\delta^q+\delta=0\}\) and
\(A_1(P_\infty)=\{[1,b,c]:c^q+c=b^{q+1}\}\).  The following lemma records
the group law and identifies the centre of \(A_1(P_\infty)\).

\begin{lemma}
\label{lem:group_law}
For \([a_i,b_i,c_i]\in A(P_\infty)\), \(i=1,2\), the product is
\[
 [a_1,b_1,c_1][a_2,b_2,c_2]
 =[a_1a_2,\ a_2b_1+b_2,\ a_2^{q+1}c_1+a_2b_2^qb_1+c_2].
\]
For \(g=[1,b,c]\) and \(h=[1,b',c']\), with
\([g,h]=ghg^{-1}h^{-1}\), the inverse and commutator are
\[
        g^{-1}=[1,-b,b^{q+1}-c],
        \qquad
        [g,h]=[1,0,b'^qb-b^qb'].
\]
Consequently,
\(Z(A_1(P_\infty))=\{[1,0,\delta]:\delta\in\Tcal\}\).
\end{lemma}

\begin{proof}
The product, inverse, and commutator formulas follow by direct
computation.  The commutator formula shows
that every \([1,0,\delta]\), with
\(\delta\in\Tcal\), is central.  Conversely, if \([1,b,c]\) is central,
then \(b'^qb-b^qb'=0\) for every \(b'\in\F_{q^2}\).  Taking \(b'=1\)
gives \(b\in\F_q\).  Choosing \(b'\in\F_{q^2}\setminus\F_q\), the
same relation becomes \(b(b'^q-b')=0\), and hence \(b=0\).  Therefore
\(c^q+c=0\), so \(c\in\Tcal\).
\end{proof}

The product formula shows that the scalar map \([a,b,c]\mapsto a\) is
a surjective homomorphism with kernel \(A_1(P_\infty)\).  Hence
\(A_1(P_\infty)\triangleleft A(P_\infty)\) and
\(A(P_\infty)/A_1(P_\infty)\simeq\F_{q^2}^*\).

The map
\(\delta\mapsto\delta^q+\delta\) is the surjective trace map from
\(\F_{q^2}\) to \(\F_q\), so \(\Tcal\) is a one-dimensional
\(\F_q\)-subspace of \(\F_{q^2}\) and \(|\Tcal|=q\).  Since
\(b^{q+1}\in\F_q\), for each \(b\in\F_{q^2}\) the solutions of
\(c^q+c=b^{q+1}\) form a coset of \(\Tcal\).  Thus
\(|A_1(P_\infty)|=q^2|\Tcal|=q^3\), and
\(A_1(P_\infty)\) is a Sylow \(p\)-subgroup of \(A\)
\cite[p.~144]{GarciaStichtenothXing2000}.

\subsection{\texorpdfstring{\(p\)}{p}-subgroup data}
\label{subsec:wild_subgroup_data}

We now attach to each subgroup \(U\le A_1(P_\infty)\) the data used in
the fixed-field construction.  By Lemma~\ref{lem:group_law}, the map
\(\pi\colon A_1(P_\infty)\to(\F_{q^2},+)\), \([1,b,c]\mapsto b\), is a
homomorphism.  Its kernel is
\(\{[1,0,\delta]:\delta\in\Tcal\}\), which equals
\(Z(A_1(P_\infty))\) by Lemma~\ref{lem:group_law}.  Put
\(V=V_U:=\pi(U)\),
\(\Delta=\Delta_U:=\{\delta\in\Tcal:[1,0,\delta]\in U\}\), and
\(U_\Delta:=U\cap\ker\pi\).  Then
\(U_\Delta=\{[1,0,\delta]:\delta\in\Delta\}\).

Both \(V\) and \(\Delta\) are additive subgroups, and hence
\(\F_p\)-subspaces.  We call them the \emph{translation space} and the
\emph{central space} of \(U\), respectively.  They are the spaces
\(V_U\) and \(W_U\) in \cite[equation~(3.4)]{GarciaStichtenothXing2000}.

To record the fibres of \(\pi|_U\), choose a set-theoretic section
\(s\colon V\to U\), \(s(b)=[1,b,c_b]\), with \(c_0=0\).

\begin{proposition}
\label{prop:fibre_decomposition}
The map
\[
 \varphi\colon V\times\Delta\longrightarrow U,\qquad
 (b,\delta)\longmapsto[1,b,c_b+\delta]
\]
is a bijection.  In particular, \(|U|=|V|\,|\Delta|\).
\end{proposition}

\begin{proof}
For \(b\in V\), the fibre of \(\pi|_U\) above \(b\) is the left coset
\(U_\Delta s(b)\).  By the product formula in Lemma~\ref{lem:group_law},
\( U_\Delta s(b)=\{[1,b,c_b+\delta]:\delta\in\Delta\} \).
These fibres partition \(U\), so \(\varphi\) is a bijection and
\(|U|=|V|\,|\Delta|\).
\end{proof}

\begin{lemma}
\label{lem:closure_identities}
For all \(b,b'\in V\), the following hold.
\begin{enumerate}
\item \(c_{b+b'}\equiv c_b+c_{b'}+b'^qb\pmod{\Delta}\).
\item \(b'^qb-b^qb'\in\Delta\); equivalently,
\([U,U]\subseteq U_\Delta\).
\end{enumerate}
\end{lemma}

\begin{proof}
For part~\textup{(1)}, the product formula in Lemma~\ref{lem:group_law} gives
\[
        s(b)s(b')=[1,b+b',c_b+c_{b'}+b'^qb]
        \in U\cap\pi^{-1}(b+b').
\]
Proposition~\ref{prop:fibre_decomposition} shows that its third coordinate belongs to
\(c_{b+b'}+\Delta\), giving the stated congruence.

For part~\textup{(2)}, let \(g,h\in U\) have translation coordinates
\(b,b'\), respectively.  Lemma~\ref{lem:group_law} gives
\[
        [g,h]=[1,0,b'^qb-b^qb']\in U_\Delta.
\]
Hence \(b'^qb-b^qb'\in\Delta\) for all \(b,b'\in V\), and
\([U,U]\subseteq U_\Delta\).
\end{proof}

\begin{definition}
\label{def:p_subgroup_data}
We call
\(\bigl(V_U,\Delta_U;\{c_b+\Delta_U\}_{b\in V_U}\bigr)\) the
\emph{\(p\)-subgroup data} of \(U\), abbreviated as
\((V,\Delta;c_b+\Delta)\) when \(U\) is fixed.
\end{definition}

By Proposition~\ref{prop:fibre_decomposition}, these data determine \(U\).
They are independent of the chosen section, since changing the section
only changes the representatives \(c_b\) within their cosets.

\section{Subspace polynomials associated with the subgroup data}
\label{sec:subspace_calculus}

We collect the properties of subspace polynomials used in the fixed-field
construction.

A polynomial \(L\in\F_{q^2}[T]\) is \emph{additive} if
\(L(S+T)=L(S)+L(T)\) in \(\F_{q^2}[S,T]\); equivalently,
\(L(T)=\sum_i a_iT^{p^i}\) with \(a_i\in\F_{q^2}\).
For background on additive polynomials, see Ore's paper
\cite{Ore1933} and the finite-field text of
Lidl--Niederreiter \cite[Section~3.4]{LidlNiederreiter1997}.

For a finite \(\F_p\)-subspace \(W\subseteq\F_{q^2}\), its
\emph{subspace polynomial} is
\(B_W(T)=\prod_{w\in W}(T-w)\); see
\cite[Definition~3.5]{BenSassonKopparty2012}.  By definition, \(B_W\) is
monic of degree \(|W|\) and has zero set \(W\).  The following lemma records
the additional properties used below.
\begin{lemma}
\label{lem:subspace_quotient_polynomials}
Let \(W\subseteq\F_{q^2}\) be a finite \(\F_p\)-subspace.
\begin{enumerate}
\item The polynomial \(B_W\) is additive and
\(B_W'\in\F_{q^2}^*\).
\item If \(D\subseteq W\) is an \(\F_p\)-subspace and
\(\Gamma=B_D(W)\), then \(\Gamma\) is an \(\F_p\)-subspace of
\(\F_{q^2}\), \(B_D\) induces an isomorphism \(W/D\to\Gamma\), and
\(B_\Gamma(B_D(T))=B_W(T)\).
\end{enumerate}
\end{lemma}

\begin{proof}
We prove part~\textup{(1)} by induction on \(\dim_{\F_p}W\).  For
\(W=\{0\}\), the assertion follows from
\(B_W(T)=T\).  Choose an \(\F_p\)-subspace \(W_0\subseteq W\) of
codimension one and \(w\in W\setminus W_0\).  Then
\(W=W_0\oplus\F_pw\) and
\(B_W(T)=\prod_{a\in\F_p}B_{W_0}(T-aw)\).
By the induction hypothesis, \(B_{W_0}\) is additive and hence
\(\F_p\)-linear, so
\(B_{W_0}(T-aw)=B_{W_0}(T)-aB_{W_0}(w)\).
The identity \(\prod_{a\in\F_p}(X-aY)=X^p-Y^{p-1}X\) therefore gives
\[
        B_W(T)=B_{W_0}(T)^p
        -B_{W_0}(w)^{p-1}B_{W_0}(T).
\]
Thus \(B_W\) is additive.  Since \(w\notin W_0\),
\(B_{W_0}(w)\ne0\), and differentiation gives
\(B_W'=-B_{W_0}(w)^{p-1}B_{W_0}'\in\F_{q^2}^*\).  This proves
part~\textup{(1)}.

We next prove part~\textup{(2)}.  By part~\textup{(1)}, the restriction
of \(B_D\) to \(W\) is \(\F_p\)-linear with kernel \(D\).  Hence
\(\Gamma\) is an \(\F_p\)-subspace, \(B_D\) induces an isomorphism
\(W/D\to\Gamma\), and \(|\Gamma|=|W|/|D|\).  The monic polynomial
\(B_\Gamma(B_D(T))\) has degree \(|\Gamma||D|=|W|\) and vanishes at
every element of \(W\).  Therefore
\(B_\Gamma(B_D(T))=\prod_{w\in W}(T-w)=B_W(T)\).
\end{proof}

Subspace polynomials also give coordinate polynomials for bases.

\begin{lemma}
\label{lem:coordinate_polynomials}
Let \(W\subseteq\F_{q^2}\) be a nonzero \(\F_p\)-subspace, put
\(n=\dim_{\F_p}W\), and fix a basis
\((w_1,\ldots,w_n)\).  For \(1\le r\le n\), put
\[
 W_r=\langle w_s:s\ne r\rangle_{\F_p},\qquad
 \ell_r(T)=\frac{B_{W_r}(T)}{B_{W_r}(w_r)}.
\]
Then \(\ell_r\) is additive of degree \(p^{n-1}\), and
\[
 \ell_r\!\left(\sum_{s=1}^n a_sw_s\right)=a_r
 \qquad(a_1,\ldots,a_n\in\F_p).
\]
\end{lemma}

\begin{proof}
By definition, \(B_{W_r}\) has degree \(p^{n-1}\) and zero set
\(W_r\), while Lemma~\ref{lem:subspace_quotient_polynomials}\textup{(1)}
shows that it is additive.
Thus \(B_{W_r}(w_r)\ne0\), while \(\ell_r(w_r)=1\) and
\(\ell_r(w_s)=0\) for \(s\ne r\).  The asserted coordinate formula
follows from the additivity of \(\ell_r\).
\end{proof}

\begin{corollary}
\label{cor:additive_polynomial_representation}
Every \(\F_p\)-linear map \(\varphi\colon W\to\F_{q^2}\) is represented
on \(W\) by a unique additive polynomial \(R\) of degree less than \(|W|\).
If \(W\ne\{0\}\) and \((w_1,\ldots,w_n)\) is a basis, then
\[
 R(T)=\sum_{r=1}^n\varphi(w_r)\ell_r(T),
\]
where the \(\ell_r\) are as in Lemma~\ref{lem:coordinate_polynomials}.
If \(W\subseteq\F_q\) and \(\varphi(W)\subseteq\Tcal\), then every
coefficient of \(R\) lies in \(\Tcal\).
\end{corollary}

\begin{proof}
For \(W=\{0\}\), take \(R=0\).  Otherwise,
Lemma~\ref{lem:coordinate_polynomials} shows that \(R\) is additive, has
degree at most \(p^{n-1}<|W|\), and restricts to \(\varphi\) on \(W\).  If
two additive polynomials of degree less than \(|W|\) agree on \(W\), their
difference has at least \(|W|\) roots and smaller degree, so they are equal.

Finally, suppose that \(W\subseteq\F_q\) and
\(\varphi(W)\subseteq\Tcal\).  Then \(W_r\subseteq\F_q\), so
\(\ell_r\in\F_q[T]\).  Since \(\Tcal\) is an \(\F_q\)-subspace and
\(\varphi(w_r)\in\Tcal\), every coefficient of \(R\) lies in \(\Tcal\).
\end{proof}

We now record two basic instances.  The polynomials \(T^q+T\) and
\(T^q-T\) are monic of degree \(q\) and have zero sets \(\Tcal\) and
\(\F_q\), respectively.  Hence the definition of the subspace polynomial
gives
\begin{equation}
\label{eq:two_basic_subspace_polynomials}
        B_{\Tcal}(T)=T^q+T,
        \qquad
        B_{\F_q}(T)=T^q-T.
\end{equation}

Let \(\Delta\subseteq\Tcal\) be an \(\F_p\)-subspace.  By
Lemma~\ref{lem:subspace_quotient_polynomials}(2),
\(\Gamma_\Delta:=B_\Delta(\Tcal)\) is an \(\F_p\)-subspace of
\(\F_{q^2}\).  Set \(A_\Delta:=B_{\Gamma_\Delta}\).

\begin{proposition}
\label{prop:central_complementary_factor}
The polynomial \(A_\Delta\) is the unique monic additive polynomial
satisfying
\[
        A_\Delta(B_\Delta(T))=T^q+T.
\]
Moreover, \(\deg A_\Delta=q/|\Delta|\),
\(\ker A_\Delta=\Gamma_\Delta\), and \(B_\Delta\) induces an
isomorphism \(\Tcal/\Delta\to\Gamma_\Delta\).
\end{proposition}

\begin{proof}
Apply Lemma~\ref{lem:subspace_quotient_polynomials}(2) to
\(\Delta\subseteq\Tcal\), and use
\eqref{eq:two_basic_subspace_polynomials}.  This gives the composition
identity and the quotient isomorphism.  Since
\(A_\Delta=B_{\Gamma_\Delta}\), its degree and kernel follow from the
definition of the subspace polynomial.  If \(C\in\F_{q^2}[T]\) also
satisfies \(C(B_\Delta(T))=T^q+T\), then
\((C-A_\Delta)(B_\Delta(T))=0\).  Substitution by the nonconstant polynomial
\(B_\Delta\) is injective on \(\F_{q^2}[T]\), so \(C=A_\Delta\).
\end{proof}

We also record the behaviour of subspace polynomials under scalar
multiplication, which will be used in Section~\ref{sec:cyclic_scalar_part}.

\begin{lemma}
\label{lem:scalar_action}
Let \(W\subseteq\F_{q^2}\) be a finite \(\F_p\)-subspace.  For every
\(\alpha\in \F_{q^2}^*\),
\[
        B_{\alpha W}(T)=\alpha^{|W|}B_W(\alpha^{-1}T).
\]
If \(\eta\in \F_{q^2}^*\) satisfies \(\eta W=W\), then
\(B_W(\eta T)=\eta B_W(T)\).
\end{lemma}

\begin{proof}
Put \(N=|W|\).  The first formula follows from
\[
        \alpha^{N}B_W(\alpha^{-1}T)
        =\prod_{w\in W}(T-\alpha w)
        =B_{\alpha W}(T).
\]
If \(\eta W=W\), applying the first formula with
\(\alpha=\eta^{-1}\) gives \(B_W(\eta T)=\eta^N B_W(T)\).
By Lemma~\ref{lem:subspace_quotient_polynomials}(1), \(B'_W\) is a
nonzero constant.  Differentiating therefore gives
\(\eta B_W'=\eta^N B_W'\), and hence \(\eta^N=\eta\).  The last
assertion follows.
\end{proof}

\section{Fixed subfields of \texorpdfstring{\(p\)}{p}-subgroups}
\label{sec:wild_fixed_field}

Let \(U\le A_1(P_\infty)\) have \(p\)-subgroup data
\((V,\Delta;c_b+\Delta)\) as in
Section~\ref{sec:pointed_subgroup_data}, and put \(N=|V|\).

To obtain explicit generators of \(H^U\), we first take invariants under the
central translations determined by \(\Delta\).  The quotient translation
group then acts on the resulting second coordinate; a finite weighted sum
solves the associated difference equations and gives a basis-free equation.
After establishing the geometry of this model, we choose a basis of \(V\) and
obtain a second explicit formula using only the fibre values on that basis.

\subsection{Central invariants and the induced translation action}
\label{subsec:central_invariants}

Put \(Y_\Delta=B_\Delta(y)\). The subgroup \(U_\Delta\) fixes \(x\) and
\(Y_\Delta\), since \(B_\Delta(y+\delta)=B_\Delta(y)\) for
\(\delta\in\Delta\).
The element \(y\) is a root of
\(B_\Delta(T)-Y_\Delta\), and hence
\[
\F_{q^2}(x,Y_\Delta)\subseteq H^{U_\Delta},
\qquad
[H:\F_{q^2}(x,Y_\Delta)]\le |\Delta|=[H:H^{U_\Delta}].
\]
Thus \(H^{U_\Delta}=\F_{q^2}(x,Y_\Delta)\).  Consequently,
\(U/U_\Delta\) acts on this field, and we now make this quotient action
explicit.

Every element of \(U\) lying over \(b\in V\) induces on this field the
same action as the section representative \(s(b)=[1,b,c_b]\), since two
such elements differ by an element of \(U_\Delta\), which acts trivially
on \(H^{U_\Delta}\).  Define
\(\lambda_b(T)=B_\Delta(b^qT+c_b)\).  Since
\(s(b)(y)=y+b^qx+c_b\),
\[
        s(b)(x)=x+b,\qquad
        s(b)(Y_\Delta)
        =B_\Delta(y+b^qx+c_b)
        =Y_\Delta+\lambda_b(x).
\]
Since \(\ker B_\Delta=\Delta\), the polynomial
\(\lambda_b\) depends only on the coset \(c_b+\Delta\).
Lemma~\ref{lem:closure_identities}\textup{(1)} gives, for \(b,b'\in V\),
\begin{equation}
\label{eq:lambda_compatibility}
\begin{aligned}
\lambda_{b+b'}(T)
 &=B_\Delta\bigl((b+b')^qT+c_b+c_{b'}+b'^qb\bigr)\\
 &=B_\Delta(b^qT+c_b)+B_\Delta\bigl(b'^q(T+b)+c_{b'}\bigr)\\
 &=\lambda_b(T)+\lambda_{b'}(T+b).
\end{aligned}
\end{equation}

\subsection{The correction polynomial and invariant coordinates}
\label{subsec:correction_polynomial}

To cancel the increments of \(Y_\Delta\), we seek a polynomial
\(\Theta\in\F_{q^2}[T]\) satisfying
\(\Theta(T+b)-\Theta(T)=\lambda_b(T)\) for every \(b\in V\).  For such a
polynomial, \(Y_\Delta-\Theta(x)\) is fixed by \(U/U_\Delta\).
Relation~\eqref{eq:lambda_compatibility} is the compatibility relation for
this system of difference equations.  If \(N\) were invertible, ordinary
averaging would give a solution from
\(-N^{-1}\sum_{u\in V}\lambda_u(T)\).  When \(N>1\), however, \(N=0\) in
\(\F_{q^2}\).  We therefore use the following identity.  The polynomials
\((T+u)^{N-1}/B'_V\) have sum \(1\) and are permuted by translation.

\begin{lemma}
\label{lem:translation_sum_identity}
With \(N=|V|\), the following identity holds:
\[
        \sum_{u\in V}(T+u)^{N-1}=B'_V.
\]
\end{lemma}

\begin{proof}
Lagrange interpolation of \((S+T)^{N-1}\), viewed as a polynomial in
\(S\), at the points of \(V\) gives
\[
 (S+T)^{N-1}=\sum_{b\in V}(T+b)^{N-1}
 \frac{B_V(S)}{(S-b)B'_V},
\]
where \(B_V'(b)=B'_V\ne0\) by
Lemma~\ref{lem:subspace_quotient_polynomials}\textup{(1)}.
Comparing the coefficients of \(S^{N-1}\) proves the identity.
\end{proof}

We shall also use the following elementary factorization property of
translation-invariant polynomials.

\begin{lemma}
\label{lem:translation_invariant_polynomials}
Every polynomial invariant under the translations \(T\mapsto T+b\),
\(b\in V\), has a unique expression \(Q(B_V(T))\), with
\(Q\in\F_{q^2}[S]\).
\end{lemma}

\begin{proof}
Every polynomial in \(B_V(T)\) is \(V\)-invariant, and substitution by
the nonconstant polynomial \(B_V\) is injective.  The zero polynomial has
the required form.  For the converse, let \(R\ne0\) and induct on \(\deg R\).
Since \(R(b)=R(0)\) for every \(b\in V\), the polynomial \(B_V(T)\)
divides \(R(T)-R(0)\).  If \(\deg R<|V|\), this forces \(R=R(0)\).
Otherwise write \(R(T)-R(0)=B_V(T)Q_1(T)\).  Since
\(B_V(T+b)=B_V(T)\), for every \(b\in V\) we obtain
\[
0=\bigl(R(T+b)-R(0)\bigr)-\bigl(R(T)-R(0)\bigr)
 =B_V(T)\bigl(Q_1(T+b)-Q_1(T)\bigr).
\]
Thus \(Q_1\) is \(V\)-invariant, and the induction hypothesis completes
the proof.
\end{proof}

Reducing \(R(T)=Q(B_V(T))\) modulo \(B_V(T)-S\) in
\(\F_{q^2}[S][T]\) replaces \(B_V(T)\) by \(S\).  Hence the remainder is
\(Q(S)\).  We use this observation for explicit models below.

Retain the subgroup \(U\le A_1(P_\infty)\) and the notation above.  Define
its \emph{correction polynomial} by
\begin{equation}
\label{eq:Theta}
        \Theta_U(T)=
        -\frac{1}{B'_V}
        \sum_{b\in V}\lambda_b(T)(T+b)^{N-1},
\end{equation}
and put \(X=B_V(x)\) and
\(Z=B_\Delta(y)-\Theta_U(x)\).

\begin{proposition}
\label{prop:wild_invariant_coordinates}
The polynomial \(\Theta_U\) defined by \eqref{eq:Theta} satisfies
\[
        \Theta_U(T+b)-\Theta_U(T)=\lambda_b(T)
        \qquad(b\in V).
\]
Consequently, \(X\) and \(Z\) are fixed by \(U\), and there is a unique
\(\Psi_U\in\F_{q^2}[S]\) such that
\[
        T^{q+1}-A_\Delta(\Theta_U(T))=\Psi_U(B_V(T)),
        \qquad
        A_\Delta(Z)=\Psi_U(X).
\]
Moreover, \(\Psi_U(S)\) is the remainder of
\(T^{q+1}-A_\Delta(\Theta_U(T))\) upon division by \(B_V(T)-S\).
\end{proposition}

\begin{proof}
Fix \(b\in V\).  From \eqref{eq:lambda_compatibility},
\(\lambda_u(T+b)=\lambda_{u+b}(T)-\lambda_b(T)\).  Reindexing by
\(v=u+b\) and using Lemma~\ref{lem:translation_sum_identity} gives
\begin{align*}
-B'_V\Theta_U(T+b)
 &=\sum_{u\in V}\lambda_u(T+b)(T+b+u)^{N-1}\\
 &=\sum_{v\in V}\bigl(\lambda_v(T)-\lambda_b(T)\bigr)(T+v)^{N-1}\\
 &=-B'_V\Theta_U(T)-B'_V\lambda_b(T).
\end{align*}
The difference equation and \(B_V(x+b)=B_V(x)\) show that
\(X\) and \(Z\) are \(U\)-invariant.

It remains to construct \(\Psi_U\).  For \(b\in V\), the identity
\(\Theta_U(T+b)-\Theta_U(T)=\lambda_b(T)\), together with
\(A_\Delta(B_\Delta(S))=S^q+S\) and \(c_b^q+c_b=b^{q+1}\), gives
\[
 A_\Delta(\Theta_U(T+b))-A_\Delta(\Theta_U(T))
 =bT^q+b^qT+b^{q+1}.
\]
The right-hand side is \((T+b)^{q+1}-T^{q+1}\).  Hence
\(T^{q+1}-A_\Delta(\Theta_U(T))\) is \(V\)-invariant, and
Lemma~\ref{lem:translation_invariant_polynomials} gives the unique
polynomial \(\Psi_U\).  Finally, substituting \(T=x\) in the resulting
factorization and using \(y^q+y=x^{q+1}\) gives
\(A_\Delta(Z)=y^q+y-A_\Delta(\Theta_U(x))=\Psi_U(X)\).
\end{proof}

\begin{remark}
The formula for \(\Theta_U\) also follows from
\cite[Theorem~4.30]{Jacobson1985}, applied to the Galois extension
\(\F_{q^2}(T)/\F_{q^2}(B_V(T))\), whose Galois group consists of the
translations \(T\mapsto T+b\), \(b\in V\).
\end{remark}

\subsection{The fixed field and its geometry}
\label{subsec:wild_field_geometry}

\begin{theorem}
\label{thm:wild_fixed_field}
Let \(U\le A_1(P_\infty)\) have \(p\)-subgroup data
\((V,\Delta;c_b+\Delta)\), and let \(X,Z,\Psi_U\) be the invariant
generators and polynomial defined in
Proposition~\ref{prop:wild_invariant_coordinates}.  Then
\[
        H^U=\F_{q^2}(X,Z),\qquad
        A_\Delta(Z)=\Psi_U(X).
\]
\end{theorem}

\begin{proof}
Proposition~\ref{prop:wild_invariant_coordinates} gives
\(\F_{q^2}(X,Z)\subseteq H^U\) and the relation
\(A_\Delta(Z)=\Psi_U(X)\).  The element
\(x\) is a root of \(B_V(T)-X\), and, after adjoining \(x\), the element
\(y\) is a root of \(B_\Delta(T)-Z-\Theta_U(x)\).  Therefore
\[
        [H:\F_{q^2}(X,Z)]\le |V|\,|\Delta|=|U|.
\]
Since \([H:H^U]=|U|\), the inclusion must be an equality.
\end{proof}

\begin{corollary}
\label{cor:wild_affine_model_geometry}
The polynomial
\(A_\Delta(Y)-\Psi_U(S)\in\F_{q^2}[S,Y]\) is absolutely irreducible.
With coordinate functions \(S=X\) and \(Y=Z\), its zero set defines a
smooth affine plane model of \(H^U\).  The unique place of \(H^U\) not
represented on this model is \(P_\infty^U\).
\end{corollary}

\begin{proof}
Put \(\Omega=\overline{\F}_{q^2}\).  For absolute irreducibility, it is enough
to show that \(A_\Delta(Y)-\Psi_U(X)\) is the minimal polynomial of \(Z\)
over \(\Omega(X)\).  By Theorem~\ref{thm:wild_fixed_field},
\(H^U=\F_{q^2}(X,Z)\).  The translations \(x\mapsto x+b\), \(b\in V\),
are \(|V|\) distinct \(\F_{q^2}(X)\)-automorphisms of
\(\F_{q^2}(x)\), while \(x\) satisfies \(B_V(T)-X\).  Hence
\([\F_{q^2}(x):\F_{q^2}(X)]=|V|\).  Since
\([H:\F_{q^2}(x)]=q\) and \([H:H^U]=|U|=|V|\,|\Delta|\), the tower law
gives
\[
        [H^U:\F_{q^2}(X)]
        =\frac{[H:\F_{q^2}(x)][\F_{q^2}(x):\F_{q^2}(X)]}{[H:H^U]}
        =\frac{q|V|}{|V|\,|\Delta|}
        =\frac q{|\Delta|}
        =\deg A_\Delta.
\]
Because \(H^U\subseteq H\) and \(H\) has full constant field
\(\F_{q^2}\), the same is true of \(H^U\).  Thus
\(H^U/\F_{q^2}\) is regular, and extending constants to \(\Omega\)
preserves this degree.  Therefore
\([\Omega(X,Z):\Omega(X)]=q/|\Delta|=\deg A_\Delta\).  The function
\(Z\) satisfies the monic polynomial
\(A_\Delta(Y)-\Psi_U(X)\in \Omega(X)[Y]\) of this degree, so it is the
minimal polynomial of \(Z\) over \(\Omega(X)\).  Gauss's lemma gives
irreducibility in \(\Omega[X,Y]\), and hence absolute irreducibility
over \(\F_{q^2}\).  The affine model is smooth because its partial
derivative with respect to the second coordinate is the nonzero
constant \(A'_\Delta\).

It remains to identify the places represented on this model.  The functions
\(X\) and \(Z\) have no poles away from \(P_\infty^U\), whereas \(X\)
has a pole at \(P_\infty^U\).  Thus every other place has a centre on the
affine model, whereas \(P_\infty^U\) does not.  Hence \(P_\infty^U\) is the
unique place not represented on this model.
\end{proof}

Although the explicit equation depends on the full \(p\)-subgroup data
\((V,\Delta;c_b+\Delta)\), the genus depends only on \(|V|\) and
\(|\Delta|\).  By \cite[equation~(3.8)]{GarciaStichtenothXing2000}, the genus is
\begin{equation}
\label{eq:gsx_wild_genus_value}
        g(H^U)=\frac{q(q-|\Delta|)}{2|V||\Delta|}.
\end{equation}

The pair \((\Theta_U,\Psi_U)\) is independent of the chosen fibre
representatives \(c_b\).

\begin{samepage}
\begin{lemma}
\label{lem:wild_coordinate_freedom}
If \(\widetilde{\Theta}_U\in\F_{q^2}[T]\) also satisfies
\[
 \widetilde{\Theta}_U(T+b)-\widetilde{\Theta}_U(T)=\lambda_b(T)\qquad(b\in V),
\]
then there is a unique \(h\in\F_{q^2}[S]\) such that
\[
 \widetilde{\Theta}_U(T)=\Theta_U(T)+h(B_V(T)).
\]
\end{lemma}
\end{samepage}

\begin{proof}
The polynomial \(\widetilde{\Theta}_U-\Theta_U\) is invariant under translations by
\(V\).  Lemma~\ref{lem:translation_invariant_polynomials} therefore gives a
unique \(h\in\F_{q^2}[S]\) such that
\(\widetilde{\Theta}_U-\Theta_U=h\circ B_V\).
\end{proof}

If \(\widetilde Z=B_\Delta(y)-\widetilde{\Theta}_U(x)\), then \(\widetilde Z=Z-h(X)\), and the
corresponding right-hand side is
\[
 \widetilde{\Psi}_U(S)=\Psi_U(S)-A_\Delta(h(S)).
\]
Thus \((S,Y)\mapsto(S,Y-h(S))\) identifies the corresponding affine models.
We use this freedom below to obtain a formula from basis values and later, in
Section~\ref{sec:closed_families}, to choose invariant coordinates adapted to
the subgroup under consideration.

\subsection{An \texorpdfstring{\(\F_p\)-quadratic}{Fp-quadratic} formula from basis values}
\label{subsec:quadratic_additive_normal_form}

\begin{definition}
A polynomial \(P\in\F_{q^2}[T]\) is
\emph{\(\F_p\)-quadratic} if it is an \(\F_{q^2}\)-linear combination of
monomials \(T^{p^i}\) and \(T^{p^i+p^j}\), where \(i,j\ge0\).
\end{definition}

Products of two additive polynomials are \(\F_p\)-quadratic, and this class is
closed under taking \(p^k\)-th powers and \(\F_{q^2}\)-linear combinations.
For every such \(P\), the function
\((s,t)\mapsto P(s+t)-P(s)-P(t)\) on \(\F_{q^2}^2\) is
\(\F_p\)-bilinear and symmetric, and is alternating if \(p=2\).

The correction polynomial \(\Theta_U\) of
Proposition~\ref{prop:wild_invariant_coordinates} is basis-free, but its
defining sum uses the fibre data over every \(b\in V\).  We now derive an
alternative formula using only the fibre values on a basis.  By
Lemma~\ref{lem:closure_identities} and \(\ker B_\Delta=\Delta\), the values
\(B_\Delta(c_b)\) satisfy
\[
 B_\Delta(c_{b+b'})-B_\Delta(c_b)-B_\Delta(c_{b'})
 =B_\Delta(b'^qb)=B_\Delta(b^qb')
 \qquad(b,b'\in V).
\]
We first construct an explicit \(\F_p\)-quadratic polynomial that reproduces
this identity on \(V\), and then add an additive polynomial to match the fibre
values on the chosen basis.

Fix a subgroup \(U\le A_1(P_\infty)\) with data
\((V,\Delta;c_b+\Delta)\).  We first treat the case \(V\ne\{0\}\).  Choose an
\(\F_p\)-basis \(\mathcal B=(b_1,\ldots,b_\nu)\) of \(V\).  For
\(1\le r\le\nu\), let
\[
 V_r=\langle b_s:s\ne r\rangle_{\F_p},\qquad
 \ell_r(T)=\frac{B_{V_r}(T)}{B_{V_r}(b_r)}
\]
be the polynomials from
Lemma~\ref{lem:coordinate_polynomials}, and put
\(g_{rs}=B_\Delta(b_r^qb_s)\).  The first sum below combines the
\(\F_p\)-linear maps \(T\mapsto B_\Delta(b_r^qT)\) through the coordinate
polynomials \(\ell_r\).  Its difference under addition by \(b\in V\) contains
an additional mixed term, which the quadratic correction removes.  Define
\[
 Q_{\mathcal B}(T)=
 \sum_{r=1}^{\nu}\ell_r(T)B_\Delta(b_r^qT)+
 \begin{cases}
 -\dfrac12\displaystyle\sum_{r,s=1}^{\nu}
 g_{rs}\ell_r(T)\ell_s(T),&p\ne2,\\[2mm]
 \displaystyle\sum_{1\le r<s\le\nu}
 g_{rs}\ell_r(T)\ell_s(T),&p=2.
 \end{cases}
\]

\begin{lemma}
\label{lem:quadratic_primitive}
Assume that \(V\ne\{0\}\) and retain the preceding notation.  The matrix
\((g_{rs})\) is symmetric and, when \(p=2\), has zero diagonal.  The
polynomial \(Q_{\mathcal B}\) is
\(\F_p\)-quadratic and satisfies
\begin{equation}
\label{eq:quadratic_primitive_difference}
 Q_{\mathcal B}(T+b)-Q_{\mathcal B}(T)-Q_{\mathcal B}(b)
 =B_\Delta(b^qT)\qquad(b\in V).
\end{equation}
\end{lemma}

\begin{proof}
By Lemma~\ref{lem:closure_identities}\textup{(2)},
\(b_s^qb_r-b_r^qb_s\in\Delta=\ker B_\Delta\), so
\(g_{rs}=g_{sr}\).  If \(p=2\), Lemma~\ref{lem:group_law} gives
\([1,b_r,c_{b_r}]^2=[1,0,b_r^{q+1}]\in U\), so
\(b_r^{q+1}\in\Delta\) and \(g_{rr}=0\).

Each summand in \(Q_{\mathcal B}\) is a scalar multiple of a product of
additive polynomials, so \(Q_{\mathcal B}\) is \(\F_p\)-quadratic.  Let
\(b=\sum_ra_rb_r\), with \(a_r\in\F_p\).  Lemma~\ref{lem:coordinate_polynomials}
gives \(\ell_r(b)=a_r\).  Moreover,
\(B_\Delta(b_r^qb)=\sum_sg_{rs}a_s\).  Put
\(\partial_bF(T)=F(T+b)-F(T)-F(b)\).  For each \(r\), direct computation gives
\[
 \partial_b\!\left(\ell_r(T)B_\Delta(b_r^qT)\right)
 =a_rB_\Delta(b_r^qT)+\ell_r(T)B_\Delta(b_r^qb).
\]
Summing over \(r\), using
\(\sum_ra_rB_\Delta(b_r^qT)=B_\Delta(b^qT)\), and expanding the
quadratic correction give
\[
\begin{aligned}
 \partial_b\!\left(\sum_r\ell_r(T)B_\Delta(b_r^qT)\right)
 &=B_\Delta(b^qT)+\sum_{r,s}g_{rs}a_s\ell_r(T),\\
 \partial_b\!\left(-\frac12\sum_{r,s}g_{rs}\ell_r\ell_s\right)(T)
 &=-\sum_{r,s}g_{rs}a_s\ell_r(T) \qquad (p\ne2),\\
 \partial_b\!\left(\sum_{r<s}g_{rs}\ell_r\ell_s\right)(T)
 &=\sum_{r,s}g_{rs}a_s\ell_r(T) \qquad (p=2).
\end{aligned}
\]
The formulas for the two quadratic correction terms use symmetry of
\((g_{rs})\); the characteristic-two formula also uses the zero diagonal of
\((g_{rs})\).
Thus the mixed terms cancel when \(p\ne2\), while in characteristic two they
occur twice and vanish.  This proves
\eqref{eq:quadratic_primitive_difference}.
\end{proof}

We now insert the fibre values.  If \(V=\{0\}\), take
\(\mathcal B=\varnothing\) and put \(\Theta_{\mathcal B}=0\).
If \(V\ne\{0\}\), set
\begin{equation}
\label{eq:explicit_qa_theta}
 \Theta_{\mathcal B}(T)=Q_{\mathcal B}(T)+
 \sum_{r=1}^{\nu}
 \bigl(B_\Delta(c_{b_r})-Q_{\mathcal B}(b_r)\bigr)
 \ell_r(T).
\end{equation}

Since \(\ker B_\Delta=\Delta\), the coefficients in
\eqref{eq:explicit_qa_theta} depend only on the cosets \(c_{b_r}+\Delta\).
Set \(F_{\mathcal B}(T)=T^{q+1}-A_\Delta(\Theta_{\mathcal B}(T))\),
\(X=B_V(x)\), and
\(Z_{\mathcal B}=B_\Delta(y)-\Theta_{\mathcal B}(x)\).

To show that the resulting quotient polynomial remains
\(\F_p\)-quadratic, we use the following lemma, applied below with \(B=B_V\).

\begin{lemma}
\label{lem:quadratic_additive_descent}
Let \(B\in\F_{q^2}[T]\) be monic and additive, and let
\(P\in\F_{q^2}[S]\).  If \(P(B(T))\) is \(\F_p\)-quadratic,
then \(P\) is \(\F_p\)-quadratic.
\end{lemma}

\begin{proof}
Since \(B(0)=0\), the hypothesis gives \(P(0)=0\).  If \(P=0\),
there is nothing to prove.  Assume \(P\ne0\) and proceed by induction on
\(\deg P\).  Let \(cS^d\) be the leading
term of \(P\), and write \(\deg B=p^r\).  The leading term of
\(P(B(T))\) is \(cT^{dp^r}\).  Since \(P(B(T))\) is
\(\F_p\)-quadratic, the sum of the base-\(p\) digits of \(dp^r\) is at
most two.  Multiplication by \(p^r\) only shifts the base-\(p\)
expansion, so the same holds for \(d\).  Thus \(d=p^i\) or
\(d=p^i+p^j\) for suitable \(i,j\), and \(B(T)^d\) is either a
Frobenius power of \(B(T)\) or a product of two such powers.  It is
therefore \(\F_p\)-quadratic.  Consequently,
\(P(B(T))-cB(T)^d\) is \(\F_p\)-quadratic, and the induction
hypothesis applies to \(P(S)-cS^d\).
\end{proof}

\begin{samepage}
\begin{theorem}
\label{thm:quadratic_additive_normal_form}
With the preceding notation, the following hold.
\begin{enumerate}
\item If \(V\ne\{0\}\), then
\begin{equation}
\label{eq:explicit_qa_difference}
 \Theta_{\mathcal B}(T+b)-\Theta_{\mathcal B}(T)
 =B_\Delta(b^qT+c_b)\qquad(b\in V).
\end{equation}

\item There is a unique \(\Psi_{\mathcal B}\in\F_{q^2}[S]\) such that
\(F_{\mathcal B}(T)=\Psi_{\mathcal B}(B_V(T))\).  Both
\(\Theta_{\mathcal B}\) and \(\Psi_{\mathcal B}\) are
\(\F_p\)-quadratic, and
\[
 H^U=\F_{q^2}(X,Z_{\mathcal B}),\qquad
 A_\Delta(Z_{\mathcal B})=\Psi_{\mathcal B}(X).
\]
\end{enumerate}
\end{theorem}
\end{samepage}

\begin{proof}
If \(V=\{0\}\), then \(U=U_\Delta\), \(B_V(T)=T\), and
\(\Theta_{\mathcal B}=0\).  Part~\textup{(2)} holds with
\(\Psi_{\mathcal B}(S)=S^{q+1}\), and the fixed-field equality follows
from \(A_\Delta(B_\Delta(y))=y^q+y=x^{q+1}\).  Assume henceforth that
\(V\ne\{0\}\).

By Lemma~\ref{lem:quadratic_primitive}, part~\textup{(1)} reduces to
showing that \(\Theta_{\mathcal B}(b)=B_\Delta(c_b)\) for every
\(b\in V\).  Put \(d_b=B_\Delta(c_b)-Q_{\mathcal B}(b)\).  For
\(b,b'\in V\), Lemma~\ref{lem:closure_identities}\textup{(1)} and
\eqref{eq:quadratic_primitive_difference} give
\[
 d_{b+b'}-d_b-d_{b'}
 =B_\Delta(b'^qb)-B_\Delta(b^qb')=0,
\]
where the last equality follows from
Lemma~\ref{lem:closure_identities}\textup{(2)} and
\(\ker B_\Delta=\Delta\).  Hence \(d:V\to\F_{q^2}\) is
\(\F_p\)-linear.  Since the \(\ell_r\) are the coordinate functions
associated with \(\mathcal B\), we have
\(d_b=\sum_r d_{b_r}\ell_r(b)\).  Substituting this equality into
\eqref{eq:explicit_qa_theta} and using
Lemma~\ref{lem:quadratic_primitive} gives
\eqref{eq:explicit_qa_difference}.

By part~\textup{(1)} and Lemma~\ref{lem:wild_coordinate_freedom}, there
is \(h\in\F_{q^2}[S]\) such that
\(\Theta_{\mathcal B}=\Theta_U+h\circ B_V\).
Then \(Z_{\mathcal B}=Z-h(X)\), so
\(H^U=\F_{q^2}(X,Z_{\mathcal B})\).
Setting \(\Psi_{\mathcal B}(S)=\Psi_U(S)-A_\Delta(h(S))\) gives
\(F_{\mathcal B}=\Psi_{\mathcal B}\circ B_V\).
Substituting \(T=x\) yields the stated equation, and uniqueness of
\(\Psi_{\mathcal B}\) follows because composition with \(B_V\) is injective.

Finally, \eqref{eq:explicit_qa_theta} shows that
\(\Theta_{\mathcal B}\) is \(\F_p\)-quadratic.  Since \(A_\Delta\) is additive,
\(F_{\mathcal B}(T)=T^{q+1}-A_\Delta(\Theta_{\mathcal B}(T))\) is also
\(\F_p\)-quadratic.  Lemma~\ref{lem:quadratic_additive_descent} gives the
same property for \(\Psi_{\mathcal B}\).
\end{proof}

Reducing
\(F_{\mathcal B}(T)=\Psi_{\mathcal B}(B_V(T))\) modulo \(B_V(T)-S\) shows
that \(\Psi_{\mathcal B}(S)\) is the remainder of \(F_{\mathcal B}(T)\).
Lemma~\ref{lem:wild_coordinate_freedom} and
Corollary~\ref{cor:wild_affine_model_geometry} show that
\(A_\Delta(Y)-\Psi_{\mathcal B}(S)\) is absolutely irreducible and defines
a smooth affine model of \(H^U\).  The only place not represented on this
model is \(P_\infty^U\).

We conclude this section by passing from subgroups of
\(A_1(P_\infty)\) to arbitrary \(p\)-subgroups of \(A\).

\begin{corollary}
\label{cor:all_p_subgroups_up_to_conjugacy}
Every \(p\)-subgroup \(G\le A\) is conjugate to some
\(U\le A_1(P_\infty)\).  Consequently, \(H^G\) is
isomorphic to the function field of the model in
Theorem~\ref{thm:wild_fixed_field} and, after choosing a basis of the
translation space, to the function field of the model in
Theorem~\ref{thm:quadratic_additive_normal_form}.
\end{corollary}

\begin{proof}
By the Sylow theorems, \(G\) is contained in a conjugate of
\(A_1(P_\infty)\).  Hence there is \(\rho\in A\) such that
\(U:=\rho G\rho^{-1}\le A_1(P_\infty)\).  Then \(\rho(H^G)=H^U\), and the
conclusion follows from Theorems~\ref{thm:wild_fixed_field}
and~\ref{thm:quadratic_additive_normal_form}.
\end{proof}

\section{Fixed subfields of arbitrary subgroups of the decomposition group}
\label{sec:cyclic_scalar_part}

We now treat arbitrary subgroups of \(A(P_\infty)\).  We first place a
subgroup with nontrivial scalar image in diagonal normal form and write \(U\)
for its normal Sylow \(p\)-subgroup.  We then take cyclic invariants of the two
generators of \(H^U\) and transport them back by conjugation.

\subsection{A diagonal normal form for subgroups with nontrivial scalar image}
\label{subsec:diagonal_normal_form}

The following lemma gives an explicit conjugation of a finite-order element
whose scalar coordinate has the same order.

\begin{lemma}
\label{lem:explicit_tame_diagonalization}
Let \(\gamma=[a,b,c]\in A(P_\infty)\) have order \(m>1\), and assume
that \(a\) also has order \(m\).  Put \(u_0=b/(1-a)\).  If
\(a^{q+1}\ne1\), set
\[
        v_0=-\frac{c+a(1-a^q)u_0^{q+1}}{a^{q+1}-1};
\]
if \(a^{q+1}=1\), choose any \(v_0\in\F_{q^2}\) satisfying
\(v_0^q+v_0=u_0^{q+1}\).  Then
\(\rho=[1,u_0,v_0]\in A_1(P_\infty)\) and
\[
        \rho\gamma\rho^{-1}=[a,0,0].
\]
\end{lemma}

\begin{proof}
We first verify that \(\rho\) belongs to \(A_1(P_\infty)\).  When
\(a^{q+1}=1\), the required \(v_0\) exists because
\(u_0^{q+1}\in\F_q\) and the trace map \(S\mapsto S^q+S\) from
\(\F_{q^2}\) onto \(\F_q\) is surjective.  When \(a^{q+1}\ne1\), put
\(\kappa=c+a(1-a^q)u_0^{q+1}\).  Since
\(c^q+c=b^{q+1}\), \(b=(1-a)u_0\), and
\(u_0^{q+1}\in\F_q\),
\[
\begin{aligned}
\kappa^q+\kappa
&=c^q+c+\bigl(a^q(1-a)+a(1-a^q)\bigr)u_0^{q+1}\\
&=\bigl((1-a)^{q+1}+a^q(1-a)+a(1-a^q)\bigr)u_0^{q+1}\\
&=(1-a^{q+1})u_0^{q+1}.
\end{aligned}
\]
Consequently, the specified value of \(v_0\) satisfies
\[
        v_0^q+v_0
        =-\frac{\kappa^q+\kappa}{a^{q+1}-1}
        =u_0^{q+1}.
\]
Thus \(\rho=[1,u_0,v_0]\in A_1(P_\infty)\) in both cases.

We now compute the conjugate.  Lemma~\ref{lem:group_law} and
\(b=(1-a)u_0\) give
\[
        \rho\gamma\rho^{-1}
        =[a,0,(a^{q+1}-1)v_0+c+a(1-a^q)u_0^{q+1}].
\]
If \(a^{q+1}\ne1\), the third coordinate is zero by the definition of
\(v_0\).  Suppose that \(a^{q+1}=1\), and denote the third coordinate
by \(\delta_0\).  Then
\[
        [a,0,\delta_0]^n=[a^n,0,n\delta_0]\qquad(n\ge1),
\]
so \((\rho\gamma\rho^{-1})^m=[1,0,m\delta_0]=1\).  Since
\(m\mid q^2-1\), we have \(p\nmid m\), and hence \(\delta_0=0\).
\end{proof}

Let \(G_0\le A(P_\infty)\), put
\(U_0=G_0\cap A_1(P_\infty)\), and set \(m=[G_0:U_0]\).  The scalar map
identifies \(G_0/U_0\) with a subgroup of the cyclic group \(\F_{q^2}^*\),
so \(m\mid q^2-1\) and \((|U_0|,m)=1\).  When \(m>1\), choose
\(\widetilde\gamma\in G_0\) whose image generates \(G_0/U_0\), and put
\(\gamma=\widetilde\gamma^{|U_0|}=[a,b,c]\).

\begin{proposition}
\label{prop:normal_form}
With the preceding notation, \(G_0=U_0\) if \(m=1\).  If \(m>1\), then
\(\gamma\) and \(a\) both have order \(m\),
\(G_0=U_0\rtimes\langle\gamma\rangle\), and the element
\(\rho=[1,u_0,v_0]\) of
Lemma~\ref{lem:explicit_tame_diagonalization} satisfies
\[
        \rho G_0\rho^{-1}
        =(\rho U_0\rho^{-1})\rtimes\langle[a,0,0]\rangle.
\]
\end{proposition}

\begin{proof}
If \(m=1\), then \(G_0=U_0\).  Assume \(m>1\).  The subgroup \(U_0\) is
normal in \(G_0\), being the kernel of the scalar map restricted to
\(G_0\).  Since \(\widetilde\gamma^m\in U_0\), Lagrange's theorem gives
\(\gamma^m=(\widetilde\gamma^m)^{|U_0|}=1\).  On the other hand, the
coset \(\gamma U_0=(\widetilde\gamma U_0)^{|U_0|}\) has order \(m\),
since \((|U_0|,m)=1\).  Therefore
\[
        m=\operatorname{ord}(\gamma U_0)
        \mid\operatorname{ord}(\gamma)\mid m,
\]
so \(\gamma\) has order \(m\).  Under the natural identification of
\(G_0/U_0\) with the image of the scalar map, the coset \(\gamma U_0\)
corresponds to \(a\), so \(a\) also has order \(m\).

Thus \(\langle\gamma\rangle\) maps isomorphically onto \(G_0/U_0\), giving
\(G_0=U_0\rtimes\langle\gamma\rangle\).
Lemma~\ref{lem:explicit_tame_diagonalization} gives
\(\rho\gamma\rho^{-1}=[a,0,0]\), and conjugation by \(\rho\) yields the
stated form.
\end{proof}

To apply the construction of Section~\ref{sec:wild_fixed_field} to the
normalized Sylow \(p\)-subgroup, we record how its subgroup data change
under conjugation.

For \(m=1\), put \(\rho=1\) and \(U=U_0\).  For \(m>1\), retain the
element \(\rho\) of Proposition~\ref{prop:normal_form} and put
\(U=\rho U_0\rho^{-1}\).  When \(m>1\),
Lemma~\ref{lem:group_law} gives
\[
        \rho[1,\beta,c]\rho^{-1}
        =[1,\beta,c+\beta^qu_0-\beta u_0^q].
\]
The second coordinate is unchanged, while setting \(\beta=0\) shows
that the central subgroup is fixed pointwise.  Thus
\(U=\rho U_0\rho^{-1}\) has the same translation and central spaces
as \(U_0\).  If
\(\beta\mapsto[1,\beta,\widetilde c_\beta]\) is a section for \(U_0\),
then \(\beta\mapsto[1,\beta,c_\beta]\) is a section for \(U\), where
\(c_\beta=\widetilde c_\beta+\beta^qu_0-\beta u_0^q\).

\begin{lemma}
\label{lem:tame_compatibility}
Let \(U\le A_1(P_\infty)\) have \(p\)-subgroup data
\((V,\Delta;c_b+\Delta)\), and let \(a\in\F_{q^2}^*\).  Then
\([a,0,0]\) normalizes \(U\) if and only if
\[
        aV=V,
        \qquad
        a^{q+1}\Delta=\Delta,
        \qquad
        c_{ab}\equiv a^{q+1}c_b\pmod{\Delta}
        \quad (b\in V).
\]
\end{lemma}

\begin{proof}
For \(b\in V\) and \(\delta\in\Delta\),
\[
        [a,0,0]^{-1}[1,b,c_b+\delta][a,0,0]
        =[1,ab,a^{q+1}(c_b+\delta)].
\]
Hence \([a,0,0]\) normalizes \(U\) if and only if \(aV=V\) and
\(a^{q+1}(c_b+\Delta)=c_{ab}+\Delta\) for every \(b\in V\).  This equality is
equivalent to \(a^{q+1}\Delta=\Delta\) and
\(c_{ab}\equiv a^{q+1}c_b\pmod\Delta\).
\end{proof}

\subsection{Scalar action on the \texorpdfstring{\(p\)}{p}-subgroup invariants}
\label{subsec:scalar_action_invariants}

Let \(G=U\rtimes\langle\gamma\rangle\) be in the diagonal normal form of
Proposition~\ref{prop:normal_form}, where \(\gamma=[a,0,0]\) and \(a\)
has order \(m>1\).  If \((V,\Delta;c_b+\Delta)\) is the \(p\)-subgroup
data of \(U\), Lemmas~\ref{lem:tame_compatibility}
and~\ref{lem:scalar_action} give the first two of the following scaling
identities:
\begin{equation}
\label{eq:scalar_scaling_identities}
\begin{aligned}
B_V(aS)&=aB_V(S),
&B_\Delta(a^{q+1}S)&=a^{q+1}B_\Delta(S),\\
a^{q+1}\Gamma_\Delta&=\Gamma_\Delta,
&A_\Delta(a^{q+1}S)&=a^{q+1}A_\Delta(S).
\end{aligned}
\end{equation}
Here the third identity follows from the second and
\(a^{q+1}\Tcal=\Tcal\), since \(a^{q+1}\in\F_q^*\), and the fourth is
Lemma~\ref{lem:scalar_action} applied to \(\Gamma_\Delta\).

\begin{lemma}
\label{lem:tame_equivariance}
With \(X,Z,\Psi_U\) as in Theorem~\ref{thm:wild_fixed_field}, the
correction polynomial \(\Theta_U\) of
Proposition~\ref{prop:wild_invariant_coordinates} satisfies
\begin{equation}
\label{eq:theta_scalar_scaling}
        \Theta_U(aS)=a^{q+1}\Theta_U(S).
\end{equation}
Consequently,
\[
 \gamma(X)=aX,\qquad
 \gamma(Z)=a^{q+1}Z,\qquad
 \Psi_U(aS)=a^{q+1}\Psi_U(S).
\]
\end{lemma}

\begin{proof}
Put \(N=|V|\).  For \(b\in V\), Lemma~\ref{lem:tame_compatibility}
gives \(c_{ab}\equiv a^{q+1}c_b\pmod\Delta\).  Together with
\eqref{eq:scalar_scaling_identities}, this congruence gives
\[
 \lambda_{ab}(aS)
 =B_\Delta\bigl((ab)^qaS+c_{ab}\bigr)
 =B_\Delta\bigl(a^{q+1}(b^qS+c_b)\bigr)
 =a^{q+1}\lambda_b(S).
\]
Since \(aV=V\), Lemma~\ref{lem:scalar_action} gives \(a^N=a\), and hence
\(a^{N-1}=1\).  Thus
\[
\begin{aligned}
\Theta_U(aS)
 &=-\frac1{B'_V}\sum_{b\in V}
   \lambda_b(aS)(aS+b)^{N-1}\\
 &=-\frac1{B'_V}\sum_{u\in V}
   \lambda_{au}(aS)(aS+au)^{N-1}\\
 &=-\frac{a^{q+1}a^{N-1}}{B'_V}
   \sum_{u\in V}\lambda_u(S)(S+u)^{N-1}\\
 &=a^{q+1}\Theta_U(S).
\end{aligned}
\]
The first two identities in \eqref{eq:scalar_scaling_identities} and the
identity \eqref{eq:theta_scalar_scaling} give the actions on \(X\) and \(Z\).
Applying \(\gamma\)
to \(A_\Delta(Z)=\Psi_U(X)\) and using the fourth identity in
\eqref{eq:scalar_scaling_identities} gives
\(\Psi_U(aX)=a^{q+1}\Psi_U(X)\).  Since \(X\) is transcendental over
\(\F_{q^2}\), this is the stated polynomial identity.
\end{proof}

\subsection{The fixed field and transport through conjugation}
\label{subsec:scalar_fixed_field}

The scalar weights of \(X\) and \(Z\) are \(1\) and \(q+1\).  For the
cyclic action of order \(m\), the second weight depends only on the residue
class of \(q+1\) modulo \(m\).  Choosing its least nonnegative representative
gives the exponents in the equation below.

Let \(G=U\rtimes\langle\gamma\rangle\), where \(\gamma=[a,0,0]\), the
scalar \(a\) has order \(m>1\), and \(U\) has \(p\)-subgroup data
\((V,\Delta;c_b+\Delta)\).  Let \(X,Z,\Psi_U\) be as in
Theorem~\ref{thm:wild_fixed_field}; thus
\[
 H^U=\F_{q^2}(X,Z),\qquad A_\Delta(Z)=\Psi_U(X).
\]
Let \(0\le r<m\) satisfy \(r\equiv q+1\pmod m\), put
\(T=X^m\), \(W=Z/X^r\), and write
\[
 A_\Delta(S)=\sum_i\alpha_iS^{p^i},\qquad
 \Psi_U(S)=\sum_j\psi_jS^j.
\]
Set \(I_\Delta=\{i\ge0:\alpha_i\ne0\}\) and
\(J_U=\{j\ge0:\psi_j\ne0\}\).

\begin{theorem}
\label{thm:tame_fixed_field}
With the preceding notation, the following hold.
\begin{enumerate}
\item The fixed field is \(H^G=\F_{q^2}(T,W)\).

\item The generators \(T,W\) satisfy
\[
        \sum_{i\in I_\Delta}\alpha_iW^{p^i}T^{r(p^i-1)/m}
        =\sum_{j\in J_U}\psi_jT^{(j-r)/m},
\]
with nonnegative integral exponents.  Let \(F(S,Y)\) be the difference
of the two sides after replacing \(T,W\) by \(S,Y\).  Then \(F\) is
absolutely irreducible, and \(F(S,Y)=0\) is a smooth affine model of
\(H^G\).
\end{enumerate}
\end{theorem}

\begin{proof}
Lemma~\ref{lem:tame_equivariance} shows that \(T=X^m\) and \(W=Z/X^r\)
are fixed by \(\gamma\).  Hence \(\F_{q^2}(T,W)\subseteq H^G\).  Moreover,
\(H^U=\F_{q^2}(T,W)(X)\), because \(Z=WX^r\) and \(X^m=T\), so
\([H^U:\F_{q^2}(T,W)]\le m=[H^U:H^G]\).  Therefore
\(H^G=\F_{q^2}(T,W)\).

For each \(i\in I_\Delta\), comparison of coefficients in the fourth identity
of \eqref{eq:scalar_scaling_identities} gives
\(a^{(q+1)(p^i-1)}=1\).  For each \(j\in J_U\), the last identity in
Lemma~\ref{lem:tame_equivariance} gives \(a^{j-(q+1)}=1\).  Since \(a\) has
order \(m\) and \(r\equiv q+1\pmod m\),
\[
 m\mid r(p^i-1)\quad(i\in I_\Delta),
 \qquad
 j\equiv r\pmod m\quad(j\in J_U).
\]
The exponents on the left are nonnegative.  If \(j\in J_U\), then
\(j\ge0\), \(0\le r<m\), and \(j\equiv r\pmod m\), so \(j\ge r\).
Substituting \(Z=WX^r\) into \(A_\Delta(Z)=\Psi_U(X)\), dividing by
\(X^r\), and using \(T=X^m\) gives the equation in part~\textup{(2)}.

It remains to prove the geometric assertions in part~\textup{(2)}.
Put \(\Omega=\overline{\F}_{q^2}\).  By
Corollary~\ref{cor:wild_affine_model_geometry},
\(A_\Delta(Y)-\Psi_U(X)\) is irreducible in \(\Omega(X)[Y]\).
The identity
\[
 F(X^m,Y)=X^{-r}\bigl(A_\Delta(X^rY)-\Psi_U(X)\bigr)
\]
therefore shows that \(F(X^m,Y)\) is irreducible in \(\Omega(X)[Y]\),
since \(Y\mapsto X^rY\) is an invertible change of variable.
It follows that \(F(T,Y)\) is irreducible in \(\Omega(T)[Y]\).
Since the coefficient of \(Y\) in \(F(S,Y)\) is the nonzero constant
\(A'_\Delta\), Gauss's lemma gives irreducibility in \(\Omega[S,Y]\), and
hence absolute irreducibility.
Finally, \(\partial F/\partial Y=A'_\Delta\ne0\), and the affine model is
smooth.
\end{proof}

\begin{remark}
\label{rem:tame_fixed_field_transfer}
The proof of Theorem~\ref{thm:tame_fixed_field} uses only that \(X,Z\)
generate \(H^U\), satisfy an absolutely irreducible equation
\(A_\Delta(Z)=\Psi(X)\), and have scalar weights \(1\) and \(q+1\).
The cyclic-invariant construction and the fixed-field and geometric
conclusions of the theorem therefore remain valid for any pair of generators
with these properties.
\end{remark}

We now return to an arbitrary subgroup \(G_0\le A(P_\infty)\).  Put
\(U_0=G_0\cap A_1(P_\infty)\) and \(m=[G_0:U_0]\), choose \(\rho\) as in
Proposition~\ref{prop:normal_form}, with \(\rho=1\) when \(m=1\), and set
\(U=\rho U_0\rho^{-1}\).  Let
\(H^U=\F_{q^2}(X,Z)\), \(A_\Delta(Z)=\Psi_U(X)\), be the model of
Theorem~\ref{thm:wild_fixed_field}.  Put \((T,W)=(X,Z)\) when \(m=1\), and
use the generators \(T=X^m\), \(W=Z/X^r\) of
Theorem~\ref{thm:tame_fixed_field} when \(m>1\).

\begin{corollary}
\label{cor:all_decomposition_fixed_subfields}
With the preceding notation,
\[
 H^{G_0}=\F_{q^2}\bigl(\rho^{-1}(T),\rho^{-1}(W)\bigr).
\]
The generators \(\rho^{-1}(T),\rho^{-1}(W)\) satisfy the equation of
Corollary~\ref{cor:wild_affine_model_geometry} when \(m=1\), and that of
Theorem~\ref{thm:tame_fixed_field}\textup{(2)} when \(m>1\).
In each case, the equation is absolutely irreducible and defines a smooth
affine model of \(H^{G_0}\).
\end{corollary}

\begin{proof}
If \(m=1\), Theorem~\ref{thm:wild_fixed_field} gives
\(H^U=\F_{q^2}(T,W)\).  If \(m>1\), Theorem~\ref{thm:tame_fixed_field}
gives \(H^{\rho G_0\rho^{-1}}=\F_{q^2}(T,W)\).  Since
\(H^{G_0}=\rho^{-1}(H^{\rho G_0\rho^{-1}})\), the fixed-field description
follows.  Writing \(\rho=[1,u_0,v_0]\),
\[
 \rho^{-1}(x)=x-u_0,
 \qquad
 \rho^{-1}(y)=y-u_0^qx+u_0^{q+1}-v_0,
\]
so these generators are explicit.  The geometric properties follow from
Corollary~\ref{cor:wild_affine_model_geometry} when \(m=1\), and from
Theorem~\ref{thm:tame_fixed_field}\textup{(2)} when \(m>1\).
\end{proof}

This completes the proof of Theorem~A.

We finish with the genus of the fixed field in
Theorem~\ref{thm:tame_fixed_field} and the specialization to \(U=1\).  The
genus formula below is \cite[Theorem~4.4]{GarciaStichtenothXing2000},
rewritten in terms of \(m\), \(V\), and \(\Delta\).

\begin{corollary}
\label{cor:tame_genus}
Let \(G=U\rtimes\langle\gamma\rangle\) be as in
Theorem~\ref{thm:tame_fixed_field}, where \(\gamma=[a,0,0]\) and \(a\) has
order \(m>1\).  Let \((V,\Delta;c_b+\Delta)\) be the \(p\)-subgroup data
of \(U\), and put \(\ell=\gcd(m,q+1)\).  Then
\begin{equation}
\label{eq:tame_genus}
        g(H^G)=\frac{q-|\Delta|}{2m|\Delta|}
        \left(\frac q{|V|}-\ell+1\right).
\end{equation}
\end{corollary}

\begin{proof}
We determine the fixed places of a nonidentity scalar.  Put
\(\Omega=\overline{\F}_{q^2}\), and let \(\gamma^i\ne1\) correspond to
\(\lambda=a^i\).  By Lemma~\ref{lem:tame_equivariance},
\(\gamma^i(X)=\lambda X\) and \(\gamma^i(Z)=\lambda^{q+1}Z\).  Hence every
fixed affine place lies above \(X=0\).  The fibre
\(A_\Delta(Z)=\Psi_U(0)\) consists of \(q/|\Delta|\) points.  If
\(\lambda^{q+1}=1\), all of them are fixed.  If
\(\lambda^{q+1}\ne1\), the identity
\(\Psi_U(0)=\lambda^{q+1}\Psi_U(0)\) gives \(\Psi_U(0)=0\), and the only
fixed point in the fibre is \(Z=0\).  Together with the fixed place \(P_\infty^U\), the
numbers of fixed places are therefore \(q/|\Delta|+1\) and \(2\),
respectively.

The cyclic group \(\langle a\rangle\) contains exactly
\(\ell=\gcd(m,q+1)\) elements satisfying \(\lambda^{q+1}=1\), of which
\(\ell-1\) are nonidentity.  Since \(m\mid q^2-1\), the action of
\(\langle\gamma\rangle\) on \(H^U\) is tame.  In this case, double-counting
fixed pairs expresses the degree of the different as the sum of the numbers
of fixed places of the nonidentity elements.  The Hurwitz genus formula
\cite[Theorem~11.57]{HirschfeldKorchmarosTorres2008} therefore gives
\[
2g(H^U)-2=m\bigl(2g(H^G)-2\bigr)
 +(\ell-1)\left(\frac q{|\Delta|}+1\right)+2(m-\ell).
\]
Substituting \eqref{eq:gsx_wild_genus_value} and simplifying yields
\eqref{eq:tame_genus}.
\end{proof}

\begin{corollary}
\label{cor:pure_scalar_quotients}
Let \(a\in\F_{q^2}^*\) have order \(m>1\), put
\(\gamma_a=[a,0,0]\), and let \(0\le r<m\) satisfy
\(r\equiv q+1\pmod m\).  Set \(T=x^m\) and \(W=y/x^r\).  Then
\begin{equation}
\label{eq:pure_scalar_model}
 H^{\langle\gamma_a\rangle}=\F_{q^2}(T,W),\qquad
 W^qT^{r(q-1)/m}+W=T^{(q+1-r)/m}.
\end{equation}
Moreover,
\[
 g\bigl(H^{\langle\gamma_a\rangle}\bigr)
 =\frac{q-1}{2m}\bigl(q-\gcd(m,q+1)+1\bigr).
\]
\end{corollary}

\begin{proof}
Apply Theorem~\ref{thm:tame_fixed_field} and
Corollary~\ref{cor:tame_genus} with \(U=1\), so that
\(V=\Delta=\{0\}\), \(A_\Delta(S)=S^q+S\), and
\(\Psi_U(S)=S^{q+1}\).
\end{proof}

\begin{remark}
\label{rem:pure_scalar_gsx}
Put \(n=(q^2-1)/m\) and
\(t=W^{q-1}T^{r(q-1)/m}=y^{q-1}\).  Equation
\eqref{eq:pure_scalar_model} gives \(W(t+1)=T^{(q+1-r)/m}\), and hence
\(\F_{q^2}(T,W)=\F_{q^2}(T,t)\).  Raising
\(W(t+1)=T^{(q+1-r)/m}\) to the \((q-1)\)-st power gives
\[
        T^n=t(t+1)^{q-1},
\]
which recovers the equation in
\cite[Example~6.3]{GarciaStichtenothXing2000}.
\end{remark}

\section{Fixed fields of several families of subgroups}
\label{sec:closed_families}

We now specialize the general construction to several families for which the
fixed-field equations take a simpler form.  We begin with central subgroups,
treat abelian subgroups in odd and even characteristic, and conclude with a
nonabelian family.  When one of these subgroups is normalized by a scalar
automorphism, Theorem~\ref{thm:tame_fixed_field} gives the fixed field of the
resulting semidirect product.

\subsection{The central family}

Let \(\Delta\subseteq\Tcal\) be an \(\F_p\)-subspace and put
\(U_\Delta=\{[1,0,\delta]:\delta\in\Delta\}\).  This is the case
\(V=\{0\}\) of Theorem~\ref{thm:wild_fixed_field}.  With \(X=x\) and
\(Z=B_\Delta(y)\), we obtain
\begin{equation}
\label{eq:central}
        H^{U_\Delta}=\F_{q^2}(X,Z),\qquad
        A_\Delta(Z)=X^{q+1}.
\end{equation}

\begin{remark}
\label{rem:classical_central_models}
The case \(|\Delta|=p\) recovers, up to scaling, the central
prime-order model in
\cite[Theorem~2.1\textup{(II)(1)}]{CossidenteKorchmarosTorres2000}.
If \(q=2^f\) and \(\Delta=\F_2\), then
\[
        A_\Delta(S)=\sum_{i=0}^{f-1}S^{2^i},\qquad
        \sum_{i=0}^{f-1}Z^{2^i}=X^{q+1}.
\]
This is the Abd\'on--Torres curve \cite[equation~(2)]{AbdonTorres1999}, of
genus \(q(q-2)/4\).
The maximal function fields considered in
\cite[Theorems~3.12, 3.14 and Remark~3.15]{GarciaOzbudak2007} are, up to
isomorphism, fixed fields of central translation subgroups of Hermitian
function fields.
\end{remark}

\begin{corollary}
\label{cor:central_scalar_extension}
Let \(a\in\F_{q^2}^*\) have order \(m>1\), assume that
\(a^{q+1}\Delta=\Delta\), and put
\(G=U_\Delta\rtimes\langle\gamma_a\rangle\), where
\(\gamma_a=[a,0,0]\).  Let \(0\le r<m\) satisfy
\(r\equiv q+1\pmod m\), and write
\(A_\Delta(S)=\sum_i\alpha_iS^{p^i}\), where the sum runs over the
nonzero terms.  With \(X=x\), \(Z=B_\Delta(y)\), \(T=X^m\), and
\(W=Z/X^r\), we obtain
\[
H^G=\F_{q^2}(T,W),\qquad
\sum_i\alpha_iW^{p^i}T^{r(p^i-1)/m}
=T^{(q+1-r)/m}.
\]
All exponents in this equation are nonnegative integers.
\end{corollary}

\begin{proof}
By Lemma~\ref{lem:tame_compatibility}, the condition
\(a^{q+1}\Delta=\Delta\) implies that \(\gamma_a\) normalizes
\(U_\Delta\).  Apply Theorem~\ref{thm:tame_fixed_field} with
\(U=U_\Delta\), \(V=\{0\}\), and \(\Psi_U(S)=S^{q+1}\).
\end{proof}

\subsection{Abelian models in odd characteristic}
\label{subsec:split_mixed_gsx_realization}

Assume throughout this subsection that \(p\) is odd.  We first construct a
family of abelian \(p\)-subgroups and determine its fixed fields.  We then show
that every noncentral abelian \(p\)-subgroup of \(A_1(P_\infty)\) arises from
this construction.

For the construction, choose \(c\in\F_q^*\), and then choose
\(\varrho\in\F_{q^2}^*\) with \(\varrho^{q+1}=-2c\); this is possible
because the norm map \(\F_{q^2}^*\to\F_q^*\) is surjective.  Let
\(E\subseteq\F_q\) and \(\Delta\subseteq\Tcal\) be \(\F_p\)-subspaces.  By
Lemma~\ref{lem:subspace_quotient_polynomials}(2), \(B_E(\F_q)\) is an
\(\F_p\)-subspace.  Let \(A_E\) be its subspace polynomial.  The same lemma
and \eqref{eq:two_basic_subspace_polynomials} give
\(A_E(B_E(T))=T^q-T\) and \(\deg A_E=q/|E|\).

Let \(R\in\F_{q^2}[T]\) be an additive polynomial of degree less than
\(q\) whose coefficients lie in \(\Tcal\), and, for \(b\in E\) and
\(\epsilon\in\Delta\), set
\[
        g_{b,\epsilon}=[1,\varrho b,-cb^2+R(b)+\epsilon],
        \qquad
        G_{E,\Delta,R}=\{g_{b,\epsilon}:b\in E,\ \epsilon\in\Delta\}.
\]
We suppress \(c\) and \(\varrho\) from the notation.
For \(b\in E\) and \(\epsilon\in\Delta\), the assumptions on
\(R\), \(\Delta\), and \(\varrho\) give
\[
\bigl(-cb^2+R(b)+\epsilon\bigr)^q
 +\bigl(-cb^2+R(b)+\epsilon\bigr)
 =-2cb^2=(\varrho b)^{q+1}.
\]
Hence each \(g_{b,\epsilon}\) belongs to \(A_1(P_\infty)\).
Lemma~\ref{lem:group_law} gives
\(g_{b,\epsilon}g_{b',\epsilon'}=g_{b+b',\epsilon+\epsilon'}\), so
\(G_{E,\Delta,R}\le A_1(P_\infty)\) and
\(G_{E,\Delta,R}\simeq(E,+)\times(\Delta,+)\).  Its \(p\)-subgroup data
have translation space \(\varrho E\), central space \(\Delta\), and fibre
representatives \(c_{\varrho b}=-cb^2+R(b)\), \(b\in E\).

Put \(u=x/\varrho\) and \(v=y+cu^2-R(u)\).  Then
\[
        g_{b,\epsilon}(u)=u+b,\qquad
        g_{b,\epsilon}(v)=v+\epsilon.
\]
Accordingly, put \(X_E=B_E(u)\) and \(Z_R=B_\Delta(v)\).

\begin{theorem}
\label{thm:odd_triangular_closed_model}
Assume that \(p\) is odd, and retain the preceding notation.  Then
\[
        H^{G_{E,\Delta,R}}=\F_{q^2}(X_E,Z_R),\qquad
        A_\Delta(Z_R)=cA_E(X_E)^2+R(A_E(X_E)).
\]
Moreover,
\(g\bigl(H^{G_{E,\Delta,R}}\bigr)
=q(q-|\Delta|)/(2|E||\Delta|)\).
\end{theorem}

\begin{proof}
Since \(g_{b,\epsilon}(u)=u+b\) and
\(g_{b,\epsilon}(v)=v+\epsilon\), the functions \(X_E\) and \(Z_R\) are
fixed by \(G_{E,\Delta,R}\).  The elements \(u\) and \(v\) are roots of
\(B_E(T)-X_E\) and \(B_\Delta(T)-Z_R\), respectively, and
\(H=\F_{q^2}(u,v)\).  Hence
\([H:\F_{q^2}(X_E,Z_R)]\le |E||\Delta|=|G_{E,\Delta,R}|\), which gives
\(H^{G_{E,\Delta,R}}=\F_{q^2}(X_E,Z_R)\).

Since the coefficients of \(R\) lie in \(\Tcal\), we have
\(R(S)^q+R(S)=-R(S^q-S)\).  Therefore
\[
\begin{aligned}
A_\Delta(Z_R)=v^q+v
 &=y^q+y+c(u^{2q}+u^2)-R(u)^q-R(u)\\
 &=-2cu^{q+1}+c(u^{2q}+u^2)+R(u^q-u)\\
 &=c(u^q-u)^2+R(u^q-u)\\
 &=cA_E(X_E)^2+R(A_E(X_E)).
\end{aligned}
\]
Equation~\eqref{eq:gsx_wild_genus_value}, with \(|V|=|E|\), gives the
stated genus.
\end{proof}

\begin{corollary}
\label{cor:odd_abelian_closed_models}
Assume that \(p\) is odd, and let
\(U\le A_1(P_\infty)\) be a noncentral abelian \(p\)-subgroup.
Then there exist \(c\in\F_q^*\), \(\varrho\in\F_{q^2}^*\),
\(\F_p\)-subspaces \(E\subseteq\F_q\) and \(\Delta\subseteq\Tcal\),
and an additive
polynomial \(R\in\F_{q^2}[T]\) of degree less than \(q\), with
coefficients in \(\Tcal\), such that \(\varrho^{q+1}=-2c\) and
\begin{enumerate}
\item \(U=G_{E,\Delta,R}\), with \(G_{E,\Delta,R}\) as in
Theorem~\ref{thm:odd_triangular_closed_model};
\item \(H^U=\F_{q^2}(X_E,Z_R)\), with
\(A_\Delta(Z_R)=cA_E(X_E)^2+R(A_E(X_E))\).
\end{enumerate}
\end{corollary}

\begin{proof}
Write \((V,\Delta;c_b+\Delta)\) for the data of \(U\), and choose
\(0\ne\varrho\in V\).  Since \(U\) is abelian,
Lemma~\ref{lem:group_law} gives
\(b^q\varrho-\varrho^qb=0\) for every \(b\in V\).  Hence
\(E:=\varrho^{-1}V\) is contained in \(\F_q\).  Put
\(c=-\varrho^{q+1}/2\).

We next reconstruct the fibre cosets.  For \(b\in E\), we have
\((c_{\varrho b}+cb^2)^q+c_{\varrho b}+cb^2
=\varrho^{q+1}b^2+2cb^2=0\).  Hence
\(c_{\varrho b}+cb^2\in\Tcal\), and we may define
\(\overline R(b):=c_{\varrho b}+cb^2+\Delta\in\Tcal/\Delta\).  For
\(b,b'\in E\), Lemma~\ref{lem:closure_identities}\textup{(1)} gives
\[
\begin{aligned}
\overline R(b+b')
 &=c_{\varrho(b+b')}+c(b+b')^2+\Delta\\
 &=c_{\varrho b}+c_{\varrho b'}
   +\varrho^{q+1}bb'+c(b^2+2bb'+b'^2)+\Delta\\
 &=\overline R(b)+\overline R(b'),
\end{aligned}
\]
where the last equality uses \(\varrho^{q+1}=-2c\).  Thus
\(\overline R\) is \(\F_p\)-linear.  Choose an \(\F_p\)-linear section
of \(\Tcal\to\Tcal/\Delta\), and let \(R_0\colon E\to\Tcal\) be the
resulting lift of \(\overline R\).  By
Corollary~\ref{cor:additive_polynomial_representation}, there is a unique
additive polynomial \(R\in\F_{q^2}[T]\) of degree less than \(|E|\),
with all coefficients in \(\Tcal\), such that \(R(b)=R_0(b)\) for every
\(b\in E\).  The fibres of \(U\)
now satisfy \(c_{\varrho b}+\Delta=-cb^2+R(b)+\Delta\).
Proposition~\ref{prop:fibre_decomposition} therefore gives
\(U=G_{E,\Delta,R}\), proving part~\textup{(1)}.  Part~\textup{(2)}
follows from Theorem~\ref{thm:odd_triangular_closed_model}.
\end{proof}

Combining the central family of \eqref{eq:central},
Corollary~\ref{cor:odd_abelian_closed_models}, and
Corollary~\ref{cor:all_p_subgroups_up_to_conjugacy}, we conclude that,
in odd characteristic, the fixed fields of all abelian \(p\)-subgroups of
\(A\) are represented, up to isomorphism,
by \eqref{eq:central} and the family in
Theorem~\ref{thm:odd_triangular_closed_model}.
Corollary~\ref{cor:odd_abelian_closed_models} proves Theorem~B\textup{(1)}.

We now specialize to the subfamily \(R=0\), and write
\(G_{E,\Delta}:=G_{E,\Delta,0}\) and
\(Z:=Z_0=B_\Delta(y+cu^2)\).

\begin{corollary}
\label{cor:odd_triangular_realization}
Let \(q=p^f\) with \(p\) odd.  Every genus occurring among the fixed
subfields of \(p\)-subgroups of \(A(P_\infty)\) is realized by a member
\(\F_{q^2}(X_E,Z)\) of the \(R=0\) subfamily, with equation
\[
        A_\Delta(Z)=cA_E(X_E)^2
\]
for suitable \(E\subseteq\F_q\), \(\Delta\subseteq\Tcal\),
\(c\in\F_q^*\), and \(\varrho\in\F_{q^2}^*\) satisfying
\(\varrho^{q+1}=-2c\).
\end{corollary}

\begin{proof}
For genus zero, take \(\Delta=\Tcal\).  Then \(A_\Delta(S)=S\), and the
resulting fixed field is rational.  Every positive genus has the form
\(g=\frac12 p^{f-j}(p^{f-k}-1)\) for some integers \(0\le j\le f\) and
\(0\le k\le f-1\), by
\cite[Theorem~3.2]{GarciaStichtenothXing2000}.  Choose
\(E\subseteq\F_q\) and \(\Delta\subsetneq\Tcal\) with \(|E|=p^j\) and
\(|\Delta|=p^k\), and choose \(c\in\F_q^*\) and
\(\varrho\in\F_{q^2}^*\) with \(\varrho^{q+1}=-2c\).  For these choices,
Theorem~\ref{thm:odd_triangular_closed_model} gives the equation
\(A_\Delta(Z)=cA_E(X_E)^2\) and genus
\[
        \frac{q(q-|\Delta|)}{2|E||\Delta|}
        =\frac12p^{f-j}(p^{f-k}-1)=g. \qedhere
\]
\end{proof}

The genus in Theorem~\ref{thm:odd_triangular_closed_model} is independent
of \(R\), which records the additional fibre data needed to represent all
noncentral abelian subgroups.

\begin{remark}
\label{rem:classical_noncentral_prime_order}
Taking \(E=\F_p\), \(\Delta=\{0\}\), and \(R=0\) gives
\[
 A_E(S)=\sum_{i=0}^{f-1}S^{p^i},\qquad
 Z^q+Z=c\left(\sum_{i=0}^{f-1}X_E^{p^i}\right)^2.
\]
Up to scaling, this is the noncentral prime-order model in
\cite[Theorem~2.1\textup{(II)(2)}]{CossidenteKorchmarosTorres2000}.
For \(p=3\) and \(c=\varrho=1\), it specializes to the
third-largest-genus function field studied in
\cite[Section~2.1]{BeelenMontanucciVicino2026b}.
\end{remark}

\begin{remark}
\label{rem:missing_order_p2_family}
Assume that \(p\) is odd, \(\Delta=\{0\}\), \(R=0\), and
\(\dim_{\F_p}E=2\).  Then \(G_{E,\{0\}}\) is elementary abelian of
order \(p^2\), and
\[
 H^{G_{E,\{0\}}}=\F_{q^2}(X_E,Z),\qquad
 Z^q+Z=cA_E(X_E)^2.
\]
Moreover,
\(\bigl(\dim_{\F_p}V_{G_{E,\{0\}}},
|G_{E,\{0\}}\cap Z(A_1(P_\infty))|\bigr)=(2,1)\).
The central and mixed noncentral order-\(p^2\) types in
\cite[Result~2.11]{GattiKorchmaros2024} have the corresponding pairs
\((0,p^2)\) and \((1,p)\).  By \eqref{eq:gsx_wild_genus_value}, the
present, central, and mixed families have respective quotient genera
\(q(q-1)/(2p^2)\), \(q(q-p^2)/(2p^2)\), and \(q(q-p)/(2p^2)\).
These are distinct, so the present family is not conjugate in \(A\) to
either listed type, and its fixed fields are not
isomorphic to theirs.
\end{remark}

Before turning to characteristic two, we adjoin a diagonal scalar automorphism
to the groups in the \(R=0\) family and determine the resulting fixed fields.
The case \(E=\{0\}\) is covered by
Corollary~\ref{cor:central_scalar_extension}, so suppose
\(E\ne\{0\}\).  A scalar \(a\in\F_{q^2}^*\) normalizing
\(G_{E,\Delta}\) must preserve its translation space \(\varrho E\),
hence \(aE=E\).  For \(0\ne b\in E\subseteq\F_q\), this gives
\(a=(ab)/b\in\F_q^*\).

Let \(a\in\F_q^*\) have order \(m>1\), and assume
\(aE=E\) and \(a^{q+1}\Delta=\Delta\).  Since \(a^q=a\), the condition
\(a^{q+1}\Delta=\Delta\) is equivalent to \(a^2\Delta=\Delta\).  Put
\(\gamma_a=[a,0,0]\).  Since
\(\gamma_a^{-1}g_{b,\epsilon}\gamma_a=g_{ab,a^2\epsilon}\)
for \(b\in E\) and \(\epsilon\in\Delta\), the automorphism
\(\gamma_a\) normalizes \(G_{E,\Delta}\); put
\(\widetilde G_{E,\Delta,a}
=G_{E,\Delta}\rtimes\langle\gamma_a\rangle\).
Moreover, \(\gamma_a(u)=au\) and
\(\gamma_a(y+cu^2)=a^2(y+cu^2)\).

Write \(A_E(S)=\sum_j e_jS^{p^j}\) and
\(A_\Delta(S)=\sum_i\alpha_iS^{p^i}\), where the sums run over the
nonzero terms, and set \(T=X_E^m\) and \(W=Z/X_E^2\).  The scalar weights
of \(X_E\) and \(Z\) are \(1\) and \(2\).  From
\(A_E(aS)=aA_E(S)\), every nonzero \(e_j\) satisfies
\(m\mid p^j-1\); from \(A_\Delta(a^2S)=a^2A_\Delta(S)\), every nonzero
\(\alpha_i\) satisfies \(m\mid2(p^i-1)\).  Hence the exponents below are
nonnegative integers.

The functions \(T,W\) are fixed by \(\gamma_a\), and
\(\F_{q^2}(X_E,Z)=\F_{q^2}(T,W)(X_E)\) with \(X_E^m=T\).
The degree argument in the proof of
Theorem~\ref{thm:tame_fixed_field} therefore gives the fixed field below.
Substituting \(Z=WX_E^2\) into \(A_\Delta(Z)=cA_E(X_E)^2\) and dividing
by \(X_E^2\) gives
\[
H^{\widetilde G_{E,\Delta,a}}=\F_{q^2}(T,W),\qquad
\sum_i \alpha_iW^{p^i}T^{2(p^i-1)/m}
=
 c\left(\sum_j e_jT^{(p^j-1)/m}\right)^2.
\]
Corollary~\ref{cor:tame_genus} gives
\[
g\bigl(H^{\widetilde G_{E,\Delta,a}}\bigr)
=\frac{q-|\Delta|}{2m|\Delta|}
 \left(\frac q{|E|}-\gcd(m,2)+1\right).
\]

Together with Corollary~\ref{cor:central_scalar_extension}, this
specialization covers, up to isomorphism, the order-\(dp\)
quotients treated in \cite{DionigiGatti2025}, where \(p,d>3\) are distinct
primes.

\subsection{Characteristic two: quadratic refinements}
\label{subsec:char_two_abelian_structure}

In characteristic two, \(\Tcal=\F_q\).  We first recall the subgroup
existence criterion of \cite[Theorem~3.3]{GarciaStichtenothXing2000}.
We then parameterize the abelian subgroups with prescribed translation and
central spaces by quadratic refinements.  The values of a refinement on a
basis allow us to choose the fibre representatives required by
Theorem~\ref{thm:quadratic_additive_normal_form}.

\begin{lemma}
\label{lem:char_two_existence_criterion}
Assume that \(p=2\), and let
\(V\subseteq\F_{q^2}\) and \(\Delta\subseteq\Tcal\) be
\(\F_2\)-subspaces.  There is a subgroup \(U\le A_1(P_\infty)\) with
translation space \(V\) and central space \(\Delta\) if and only if
\[
        V^{q+1}:=\{b^{q+1}:b\in V\}\subseteq\Delta.
\]
\end{lemma}

\begin{proof}
Necessity follows from
\([1,b,c]^2=[1,0,b^{q+1}]\in U\).  Conversely, suppose that
\(V^{q+1}\subseteq\Delta\), and put \(\widehat U=\pi^{-1}(V)\),
\(U_{\Tcal}:=\ker\pi\), and
\(U_\Delta:=\{[1,0,\delta]:\delta\in\Delta\}\).  It is enough to
construct \(U\le\widehat U\) with
\(\pi(U)=V\) and \(U\cap U_{\Tcal}=U_\Delta\).

Every \([1,b,c]\in\widehat U\) satisfies
\([1,b,c]^2=[1,0,b^{q+1}]\in U_\Delta\).  Thus every element of
\(\widehat U/U_\Delta\) has order at most two, so this quotient is abelian
and is an \(\F_2\)-vector space.  The exact sequence
\[
0\longrightarrow U_{\Tcal}/U_\Delta\longrightarrow
\widehat U/U_\Delta\longrightarrow V\longrightarrow0
\]
therefore splits.  Choose a complement \(\overline U\) of
\(U_{\Tcal}/U_\Delta\) in \(\widehat U/U_\Delta\), and let \(U\) be its
inverse image in \(\widehat U\).  Then \(\overline U\to V\) is an
isomorphism, so \(\pi(U)=V\) and
\(U\cap U_{\Tcal}=U_\Delta\), as required.
\end{proof}

\begin{samepage}
\begin{corollary}
\label{cor:char_two_abelian_structure}
Assume that \(p=2\).
\begin{enumerate}
\item If \(U\le A_1(P_\infty)\) is noncentral and abelian and
\(0\ne\varrho\in V_U\), then \(V_U=\varrho E\) for an
\(\F_2\)-subspace \(E\subseteq\F_q\).

\item Given \(\varrho\in\F_{q^2}^*\), a nonzero \(\F_2\)-subspace
\(E\subseteq\F_q\), and an \(\F_2\)-subspace
\(\Delta\subseteq\F_q\), put \(E^{(2)}=\{a^2:a\in E\}\).  Then an
abelian subgroup with translation space \(\varrho E\) and central space
\(\Delta\) exists if and only if
\[
        \varrho^{q+1}E^{(2)}\subseteq\Delta.
\]
\end{enumerate}
\end{corollary}
\end{samepage}

\begin{proof}
For part~\textup{(1)}, let \(b\in V_U\).  Since \(U\) is abelian,
Lemma~\ref{lem:group_law} gives
\(\varrho^qb-b^q\varrho=0\).  Hence
\((b/\varrho)^q=b/\varrho\), so \(b/\varrho\in\F_q\).  Thus
\(E=\varrho^{-1}V_U\).

For part~\textup{(2)}, put \(V=\varrho E\).  For \(e,e'\in E\),
\((\varrho e')^q\varrho e=(\varrho e)^q\varrho e'\), so every subgroup
with translation space \(V\) is abelian.  Moreover,
\(V^{q+1}=\varrho^{q+1}E^{(2)}\), and
Lemma~\ref{lem:char_two_existence_criterion} gives the stated criterion.
\end{proof}

\begin{corollary}
\label{cor:char_two_group_structure}
Assume that \(p=2\), and let \(U\) be an abelian subgroup with translation
space \(\varrho E\) and central space \(\Delta\).  Put
\(e_E=\dim_{\F_2}E\), \(e_\Delta=\dim_{\F_2}\Delta\).  Then
\(e_E\le e_\Delta\) and
\[
        U\simeq C_4^{\,e_E}\times C_2^{\,e_\Delta-e_E}.
\]
\end{corollary}

\begin{proof}
The image of the squaring map is
\(\{[1,0,\varrho^{q+1}a^2]:a\in E\}\), which has order \(2^{e_E}\),
and \(U\) has exponent at most four.  Proposition~\ref{prop:fibre_decomposition}
gives \(|U|=2^{e_E+e_\Delta}\).  Writing
\(U\simeq C_4^r\times C_2^s\), the orders of the subgroup of squares and
of \(U\) give \(r=e_E\) and \(s=e_\Delta-e_E\).
\end{proof}

Corollary~\ref{cor:char_two_group_structure} determines the abstract group
structure.  To distinguish subgroups with the same translation and central
spaces, it remains to parameterize their fibre cosets.  Fix
\(\varrho\in\F_{q^2}^*\), a nonzero \(\F_2\)-subspace
\(E\subseteq\F_q\), and an \(\F_2\)-subspace \(\Delta\subseteq\F_q\)
satisfying \(\varrho^{q+1}E^{(2)}\subseteq\Delta\), and put
\(\mu=\varrho^{q+1}\).  The map
\((e,e')\mapsto\mu ee'+\Delta\) from \(E\times E\) to
\(\F_q/\Delta\) is \(\F_2\)-bilinear and alternating, since
\(\mu e^2\in\Delta\) for every \(e\in E\).

\begin{definition}
With \(\mu,E,\Delta\) as above, a map
\(\eta\colon E\to\F_q/\Delta\) is a \emph{quadratic refinement} of
this bilinear map if
\begin{equation}
\label{eq:char_two_refinement_identity}
\eta(e+e')+\eta(e)+\eta(e')=\mu ee'+\Delta
 \qquad(e,e'\in E).
\end{equation}
\end{definition}

Choose \(\theta\in\F_{q^2}\) with \(\theta^q+\theta=1\).  Let \(U\) be an
abelian subgroup with translation space \(\varrho E\) and central space
\(\Delta\), and write \(c_{\varrho e}+\Delta\) for its fibre over
\(\varrho e\).  Since
\((\theta\mu e^2)^q+\theta\mu e^2=(\varrho e)^{q+1}\), we have
\(c_{\varrho e}+\theta\mu e^2\in\F_q\).  Thus
\(\eta_U(e)=c_{\varrho e}+\theta\mu e^2+\Delta\) defines a map
\(E\to\F_q/\Delta\).  Conversely, for a quadratic refinement \(\eta\), put
\[
 U_\eta=\{[1,\varrho e,\theta\mu e^2+r]:
 e\in E,\ r\in\F_q,\ r+\Delta=\eta(e)\}.
\]

\begin{proposition}
\label{prop:char_two_quadratic_refinements}
The assignments \(U\mapsto\eta_U\) and \(\eta\mapsto U_\eta\) give inverse
bijections between the abelian subgroups with translation space
\(\varrho E\) and central space \(\Delta\) and the quadratic refinements
satisfying \eqref{eq:char_two_refinement_identity}.
\end{proposition}

\begin{proof}
For each \(e\in E\), the condition \(c^q+c=\mu e^2\) is equivalent to
\(c+\theta\mu e^2\in\F_q\).  Thus each fibre coset has a unique expression
\(\theta\mu e^2+(r_e+\Delta)\), with \(r_e+\Delta\in\F_q/\Delta\).
By Lemma~\ref{lem:group_law}, the corresponding set of automorphisms is
closed under multiplication if and only if
\[
 r_{e+e'}+\Delta=r_e+r_{e'}+\mu ee'+\Delta
 \qquad(e,e'\in E),
\]
which is \eqref{eq:char_two_refinement_identity} for
\(\eta(e)=r_e+\Delta\).  Hence every subgroup with the prescribed spaces
gives a quadratic refinement.

Conversely, a refinement satisfies \(\eta(0)=0\), so \(U_\eta\) is a
finite set of automorphisms containing the identity and closed under
multiplication.  It is therefore a subgroup, with translation and central
spaces \(\varrho E\) and \(\Delta\).  It is abelian by the commutator formula
in Lemma~\ref{lem:group_law}, since \(E\subseteq\F_q\).
The definitions give \(U_{\eta_U}=U\) and \(\eta_{U_\eta}=\eta\).
\end{proof}

Corollary~\ref{cor:char_two_abelian_structure}\textup{(2)} and
Proposition~\ref{prop:char_two_quadratic_refinements} show that refinements
exist for the fixed data.  The difference of two quadratic refinements is
\(\F_2\)-linear, and adding any \(\F_2\)-linear map
\(E\to\F_q/\Delta\) to a refinement gives another.  Thus, writing
\(e_E=\dim_{\F_2}E\), the refinements are determined by arbitrary values on
an \(\F_2\)-basis of \(E\), and there are exactly
\(\bigl(q/|\Delta|\bigr)^{e_E}\) abelian subgroups with the prescribed
translation and central spaces.

For a refinement \(\eta\), choose an \(\F_2\)-basis
\(e_1,\ldots,e_{e_E}\) of \(E\) and lifts
\(\widetilde\xi_i\in\F_q\) of \(\eta(e_i)\).  For the basis
\(\mathcal B=(\varrho e_1,\ldots,\varrho e_{e_E})\) of \(\varrho E\), the elements
\[
 c_{\varrho e_i}=\theta\mu e_i^2+\widetilde\xi_i,
 \qquad 1\le i\le e_E,
\]
are fibre representatives for \(U_\eta\).  They supply the data required in
\eqref{eq:explicit_qa_theta}, so
Theorem~\ref{thm:quadratic_additive_normal_form} gives an explicit
\(\F_2\)-quadratic equation for \(H^{U_\eta}\).  This completes the proof of
Theorem~B\textup{(2)}.

\subsection{The case of zero refinement}
\label{subsec:char_two_zero_refinement_family}

In the preceding notation, the zero map \(\eta\colon E\to\F_q/\Delta\)
satisfies
\eqref{eq:char_two_refinement_identity} if and only if
\[
 \mu\left\langle bb':b,b'\in E\right\rangle_{\F_2}\subseteq\Delta.
\]
Under this condition, Proposition~\ref{prop:char_two_quadratic_refinements}
gives the subgroup
\[
 U_0=\{[1,\varrho b,\theta\mu b^2+\delta]:b\in E,\ \delta\in\Delta\}.
\]
The existence criterion requires only \(\mu b^2\in\Delta\) for \(b\in E\),
whereas zero refinement requires \(\mu bb'\in\Delta\) for all \(b,b'\in E\).

Let \(q=Q^d\), where \(Q=2^e\) and \(d\ge3\) is odd, and take
\(\varrho=1\) and \(E=\Delta=\F_Q\).  Then \(\mu=1\), and the
zero-refinement condition holds because \(\F_Q\) is closed under
multiplication.

Choose \(\theta\in\F_{Q^2}\) with \(\theta^Q+\theta=1\).  Since
\(d\) is odd, we also have \(\theta^q+\theta=1\).  Write the resulting
group as
\[
        N_Q=\{[1,b,\theta b^2+\delta]:b,\delta\in\F_Q\}.
\]

For \(r\ge1\), define
\[
        L_r(S)=\sum_{i=0}^{r-1}S^{Q^i},
        \qquad
        \Phi_{d,Q}(S)=\sum_{r=1}^{d-1}L_{r+1}(S)L_r(S),
\]
and set \(X=x^Q+x\), \(Z=y^Q+y+x^{Q+1}\).
Here \(L_d\) is the \(Q\)-trace polynomial.

\begin{samepage}
\begin{theorem}
\label{thm:char_two_closed_family}
With the preceding notation, the following hold.
\begin{enumerate}
\item The set \(N_Q\) is an abelian subgroup of \(A_1(P_\infty)\) of
order \(Q^2\), isomorphic to \(C_4^{\,e}\).
\item Its fixed field has the equation
\[
        H^{N_Q}=\F_{q^2}(X,Z),
        \qquad
        L_d(Z)=\Phi_{d,Q}(X).
\]
\item Its genus is \(g(H^{N_Q})=\frac12Q^{d-1}(Q^{d-1}-1)\).
\end{enumerate}
\end{theorem}
\end{samepage}

\begin{proof}
By Proposition~\ref{prop:char_two_quadratic_refinements}, \(N_Q\) is an
abelian subgroup with translation and central spaces \(\F_Q\).
Corollary~\ref{cor:char_two_group_structure} gives
\(N_Q\simeq C_4^{\,e}\), and hence \(|N_Q|=Q^2\).  This proves
part~\textup{(1)}.

We now determine the fixed-field equation.  Here
\(B_V(S)=B_\Delta(S)=S^Q+S\).  Since \(L_d(S^Q+S)=S^q+S\),
Proposition~\ref{prop:central_complementary_factor} gives
\(A_\Delta(S)=L_d(S)\).  For the section \(c_b=\theta b^2\) and
\(b\in\F_Q\),
\[
\lambda_b(S)=B_\Delta(b^qS+\theta b^2)
=bS^Q+bS+b^2
=(S+b)^{Q+1}+S^{Q+1}.
\]
Thus the correction polynomial \(\Theta(S)=S^{Q+1}\) satisfies
\(\Theta(S+b)-\Theta(S)=\lambda_b(S)\) for every \(b\in\F_Q\).
Lemma~\ref{lem:wild_coordinate_freedom} therefore allows us to take the
invariant generators to be \(X,Z\).

It remains to identify the right-hand side.  Since
\(L_r(S^Q+S)=S^{Q^r}+S\), expanding the definition of
\(\Phi_{d,Q}\) gives
\[
\Phi_{d,Q}(S^Q+S)
=\sum_{r=1}^{d-1}\bigl(
S^{Q^{r+1}+Q^r}+S^{Q^{r+1}+1}+S^{Q^r+1}+S^2\bigr).
\]
The middle two terms telescope, while the \(S^2\)-terms sum to zero
because \(d\) is odd.  Hence
\[
\Phi_{d,Q}(S^Q+S)=S^{q+1}+L_d(S^{Q+1}).
\]
Thus
\(S^{q+1}-A_\Delta(S^{Q+1})=\Phi_{d,Q}(B_V(S))\), and
Theorem~\ref{thm:wild_fixed_field} gives the fixed field and its
equation, proving part~\textup{(2)}.  Finally,
equation~\eqref{eq:gsx_wild_genus_value}, with
\((|V|,|\Delta|)=(Q,Q)\), gives part~\textup{(3)}.
\end{proof}

For \(Q=2\), \(N_2\simeq C_4\), so
Theorem~\ref{thm:char_two_closed_family} gives an explicit model for the
quotient by a cyclic subgroup of order four fixing \(P_\infty\); the
corresponding cyclic order-four quotients are treated in
\cite[Theorem~5.1]{GattiKorchmaros2024}.

\subsection{A unitary Heisenberg subgroup}
\label{subsec:nonabelian_heisenberg_specialization}

Let \(q=Q^d\), where \(Q\) is a power of \(p\) and
\(d=2s+1\ge3\).  Consider the following subgroup, which is the
one-dimensional unitary Heisenberg group in the coordinates used here
\cite[Section~2.1, equation~(2.1)]{ImaiTsushima2023}:
\[
 U_Q=\{\sigma_{b,c}:b,c\in\F_{Q^2},\ c^Q+c=b^{Q+1}\},\qquad
 \sigma_{b,c}(x)=x+b,\quad \sigma_{b,c}(y)=y+b^Qx+c.
\]

Since \(d\) is odd, the \(q\)-Frobenius restricts to the \(Q\)-Frobenius
on \(\F_{Q^2}\).  Put
\(\Delta=\{c\in\F_{Q^2}:c^Q+c=0\}\).  Then
\(\Delta\subseteq\Tcal\) and \(B_\Delta(T)=T^Q+T\).

\begin{proposition}
\label{prop:heisenberg_group_structure}
The set \(U_Q\) is a nonabelian subgroup of \(A_1(P_\infty)\) of order
\(Q^3\).  Its centre and commutator subgroup coincide and equal
\(\{\sigma_{0,c}:c^Q+c=0\}\), a group of order \(Q\).
\end{proposition}

\begin{proof}
For \(\sigma_{b,c}\in U_Q\), the equalities \(b^q=b^Q\) and \(c^q=c^Q\)
show that \(\sigma_{b,c}\in A_1(P_\infty)\).  By
Lemma~\ref{lem:group_law}, \(U_Q\) is closed under multiplication and
inversion, and
\([\sigma_{b,c},\sigma_{b',c'}]=\sigma_{0,b'^Qb-b^Qb'}\).
Hence \(U_Q\) is a subgroup.  The trace map \(c\mapsto c^Q+c\) has kernel
of order \(Q\), so \(|U_Q|=Q^3\).
An element is central precisely when \(b=0\): for \(b\ne0\), choose
\(b'\notin b\F_Q\).  Finally, taking \(b=1\), the map \(b'\mapsto b'^Q-b'\)
has kernel \(\F_Q\) and image \(\Delta\), so every central element is a
commutator.
\end{proof}

Since
\(\sum_{i=0}^{d-1}(-1)^i(S^Q+S)^{Q^i}=S^{Q^d}+S=S^q+S\),
Proposition~\ref{prop:central_complementary_factor} gives the first identity
below; define \(\Fcal_{d,Q}\) by the second:
\begin{equation}
\label{eq:heisenberg_specialization_rhs}
 A_\Delta(T)=\sum_{i=0}^{d-1}(-1)^iT^{Q^i},\qquad
 \Fcal_{d,Q}(T)=-\sum_{0\le i\le j\le s-1}
 T^{Q^{2i}+Q^{2j+1}}.
\end{equation}
Put \(X=x^{Q^2}-x\) and \(Z=y^Q+y-x^{Q+1}\).

\begin{theorem}
\label{thm:heisenberg_specialization}
With \(A_\Delta\) and \(\Fcal_{d,Q}\) as in
\eqref{eq:heisenberg_specialization_rhs}, the fixed field of \(U_Q\) is
\[
        H^{U_Q}=\F_{q^2}(X,Z),\qquad
        A_\Delta(Z)=\Fcal_{d,Q}(X).
\]
This equation is absolutely irreducible and defines a smooth affine model.
Moreover, \(g(H^{U_Q})=\frac12Q^{d-2}(Q^{d-1}-1)\).
\end{theorem}

\begin{proof}
The translation and central spaces of \(U_Q\) are \(\F_{Q^2}\) and
\(\Delta\), so \(B_V(T)=T^{Q^2}-T\).  For a section
\(b\mapsto\sigma_{b,c_b}\),
\(B_\Delta(b^qT+c_b)=(T+b)^{Q+1}-T^{Q+1}\).  Thus \(T^{Q+1}\) is a
correction polynomial, and Lemma~\ref{lem:wild_coordinate_freedom} gives the
chosen coordinates
\(X,Z\).

It remains to compute the right-hand side.  Since \(X=x^{Q^2}-x\),
we have \(\sum_{i=0}^jX^{Q^{2i}}=x^{Q^{2j+2}}-x\) for
\(0\le j\le s-1\).  Hence
\[
 \Fcal_{d,Q}(X)=-\sum_{j=0}^{s-1}
   \bigl(x^{Q^{2j+3}}-x^{Q^{2j+1}}\bigr)
   \bigl(x^{Q^{2j+2}}-x\bigr).
\]
Expanding and telescoping give
\[
 \Fcal_{d,Q}(X)
 =x^{q+1}-\sum_{r=0}^{d-1}(-1)^r\bigl(x^{Q+1}\bigr)^{Q^r}
 =x^{q+1}-A_\Delta(x^{Q+1}).
\]
The fixed-field equation follows from Theorem~\ref{thm:wild_fixed_field};
Corollary~\ref{cor:wild_affine_model_geometry} and
Lemma~\ref{lem:wild_coordinate_freedom} give its geometric properties.
Finally, \eqref{eq:gsx_wild_genus_value}, with
\((|V|,|\Delta|)=(Q^2,Q)\), gives the genus.
\end{proof}

\section{A subfield not isomorphic to any Galois quotient of \texorpdfstring{\(H_{27}\)}{H27}}
\label{sec:non_galois_subcover}

Put \(k=\F_{27^2}=\F_{3^6}\), and write
\(H_{27}=k(x,y)\), where \(y^{27}+y=x^{28}\).  We use the \(R=0\) case of
Theorem~\ref{thm:odd_triangular_closed_model} to construct the genus-two
subfield in Theorem~C.

Since \(-2=1\) in characteristic \(3\), take \(c=\varrho=1\) in
Section~\ref{subsec:split_mixed_gsx_realization}.  Let \(E=\F_3\), and let
\(\Delta\subset k\) be the zero set of \(T^9-T^3+T\).  This polynomial is
separable and satisfies
\[
        (T^9-T^3+T)^3+(T^9-T^3+T)=T^{27}+T.
\]
Thus \(T^9-T^3+T\) divides \(T^{27}+T=B_{\Tcal}(T)\), so
\(\Delta\subseteq\Tcal\) and \(|\Delta|=9\).  The polynomials
\(B_E,B_\Delta,A_E,A_\Delta\) are
\[
\begin{aligned}
 B_E(T)&=T^3-T, & A_E(T)&=T^9+T^3+T,\\
 B_\Delta(T)&=T^9-T^3+T, & A_\Delta(T)&=T^3+T.
\end{aligned}
\]

Let \(G=G_{E,\Delta}\), and put
\(X=x^3-x\) and \(Z=B_\Delta(y+x^2)\).  By
Theorem~\ref{thm:odd_triangular_closed_model},
\[
        C:=H_{27}^G=k(X,Z),\qquad
        Z^3+Z=(X^9+X^3+X)^2,
\]
and \([H_{27}:C]=27\).  To make an order-four symmetry visible, use the
identity
\[
 (X^9+X^3+X)^2
 =(X^6+2X^4)^3+(X^6+2X^4)+2X^{10}+X^2.
\]
With \(Y=Z-X^6-2X^4\), this gives
\begin{equation}
\label{eq:q27_intermediate_model}
        C=k(X,Y),\qquad Y^3+Y=2X^{10}+X^2.
\end{equation}

\begin{proposition}
\label{prop:q27_hermitian_tower}
Choose \(i\in\F_9\subset k\) with \(i^2=-1\), and define
\(\sigma(X)=iX\) and \(\sigma(Y)=-Y\).  Then \(\sigma\) is an automorphism
of \(C\) of order \(4\).  If \(t=Y/X^2\) and \(v=X^4-t^3\), then
\[
        C^{\langle\sigma\rangle}=k(t,v),\qquad
        v^2=t^6-t+1.
\]
Consequently, \([C:k(t,v)]=4\).
\end{proposition}

\begin{proof}
Since \(i^{10}=i^2=-1\), the substitution
\((X,Y)\mapsto(iX,-Y)\) multiplies both sides of
\eqref{eq:q27_intermediate_model} by \(-1\).  Hence \(\sigma\) is an
automorphism.  Moreover, \(\sigma^4=\mathrm{id}\), whereas
\(\sigma^2(X)=-X\ne X\), so \(\sigma\) has order \(4\).  The functions
\(t\) and \(v\) are fixed by \(\sigma\).

Substituting \(Y=tX^2\) into \eqref{eq:q27_intermediate_model} and dividing by
\(X^2\) gives \(t^3X^4+t=2X^8+1\).  Substituting \(X^4=v+t^3\) then gives
\(v^2=t^6-t+1\).

Finally, \(X^4=v+t^3\) and \(Y=tX^2\) give
\(C=k(t,v,X)\), so
\([C:k(t,v)]\le4=[C:C^{\langle\sigma\rangle}]\).  Together with
\(k(t,v)\subseteq C^{\langle\sigma\rangle}\), this proves
\(C^{\langle\sigma\rangle}=k(t,v)\) and \([C:k(t,v)]=4\).
\end{proof}

Put \(D=C^{\langle\sigma\rangle}=k(t,v)\).  Since \(X=x^3-x\),
\(Z=B_\Delta(y+x^2)\), and \(Y=Z-X^6-2X^4\), the inclusion
\(D\subset H_{27}\) is given explicitly by \(t=Y/X^2\) and
\(v=X^4-t^3\).

\begin{theorem}
\label{thm:covered_not_galois}
The function field \(D\) has genus \(2\) and is not isomorphic to
\(H_{27}^J\) for any subgroup \(J\le\Aut(H_{27})\).
The extension \(H_{27}/D\) is separable of degree \(108\) and is not Galois.
\end{theorem}

\begin{proof}
Since \((t^6-t+1)'=-1\), the polynomial \(t^6-t+1\) is squarefree of degree
six.  Hence the equation \(v^2=t^6-t+1\) gives \(g(D)=2\).  By
\cite[Theorem~5.1]{MontanucciZini2017}, the genus spectrum of the Galois
quotients of \(H_{27}\) does not contain \(2\).  Hence
\(D\not\cong H_{27}^J\) for every subgroup \(J\le\Aut(H_{27})\).

The extensions \(H_{27}/C\) and \(C/D\) are Galois of degrees \(27\) and
\(4\), respectively.  Therefore \(H_{27}/D\) is separable of degree
\(108\).  If it were Galois, then \(D\) would equal \(H_{27}^J\) for
\(J=\Gal(H_{27}/D)\), a contradiction.
\end{proof}

This completes the proof of Theorem~C.

\section*{Declaration of AI use}
The authors used ChatGPT (models 5.5, 5.6, and 6) to improve the language
and correct grammatical errors in this paper.  ChatGPT also assisted in
refining some details of the arguments; the core mathematical ideas were
developed by the authors.  The authors take full responsibility for all
of its content.

\raggedbottom

\end{document}